\documentclass[preprint,12pt]{elsarticle}

\usepackage{amsmath,amssymb}
\usepackage{mathrsfs}
\usepackage{amsthm}
\usepackage{mathtools}
\usepackage{gensymb}
\usepackage{setspace}
\usepackage{extarrows}
\usepackage{geometry}
\usepackage{booktabs}
\usepackage{multirow,multicol}
\usepackage{makecell}
\usepackage{listings}
\usepackage{appendix}
\usepackage{caption,subcaption}
\usepackage{bm}
\usepackage{stack}
\usepackage{subfloat,wrapfig}
\usepackage{xcolor}
\usepackage{paralist}
\usepackage{pdfpages}
\usepackage{indentfirst}
\usepackage[normalem]{ulem}
\usepackage{algorithm}
\usepackage{algorithmic}
\usepackage[hidelinks]{hyperref}

\newtheorem{theorem}{Theorem}[section]

\newtheorem{lemma}{Lemma}[section]

\newcommand{\ie}{\, \mathrm{i}.\mathrm{e}. \, }

\begin{document}
	
\begin{frontmatter}
	
	\title{An exponential integrator multicontinuum homogenization method for fractional diffusion problem with multiscale coefficients}
	
	\author{Yifei Gao$^{1}$}
	\ead{rootgyf340@stu.xjtu.edu.cn}
	\author{Yating Wang$^{1}$\corref{cor1}}
	\ead{yatingwang@xjtu.edu.cn}
    \author{Wing Tat Leung$^{2}$}
    \ead{wtleung27@cityu.edu.hk}
    \author{Zhengya Yang$^{1}$}
    \ead{yangzy@stu.xjtu.edu.cn}
	
	\address{$^1$ School of Mathematics and Statistic, Xi'an Jiaotong University, Xi'an 710049, People's Republic of China}
    \address{$^2$ Department of Mathematics, City University of Hong Kong, Hong Kong Special Administrative Region}

	\cortext[cor1]{Corresponding author}

	\begin{abstract}
		In this paper, we present a robust and fully discretized method for solving the time fractional diffusion equation with high-contrast multiscale coefficients. We establish the homogenized equation using a multicontinuum approach and employ the exponential integrator method for time discretization. The multicontinuum upscaled model captures the physical characteristics of the solution for the high-contrast multiscale problem, including averages and gradient effects in each continuum at the coarse scale. We then use the exponential integration method for the nonlocal time fractional derivative and it can handle semilinear problem in an efficient way. Convergence analysis of the numerical scheme is provided, along with illustrative numerical examples. Our results demonstrate the accuracy, efficiency, and improved stability for varying order of fractional derivatives.
	\end{abstract}
	
	\begin{keyword}
		{Multicontinuum homogenization}; {Multiscale}; {Fractional}; {Exponential integrator}; {Convergence}
	\end{keyword}
\end{frontmatter}

\section{Introduction}


Time fractional diffusion equations are extensively utilized to describe subdiffusion phenomena in porous media. The fractional operators can be a useful way to include memory in the dynamical process. It is modeled through fractional order derivatives carrying information on its present and past states. Due to the highly heterogeneous nature of the underlying media, the multiscale coefficients in these equations often exhibit high-contrast and nonlocality characteristics. Such features are prevalent in various applications, including composite materials, groundwater modeling, and reservoir simulations. It is essential to consider the heterogeneities of porous media and the phenomena of subdiffusion, and we could achieve more accurate modeling in practical applications. However, the time fractional derivative and multiscale features of the semilinear problems pose significant challenges in simulation. 


Indeed, the presence of high-contrast/multicontinuum fields leads to the lack of smoothness in the exact solution, and the fractional order of the time derivative results in long-term memory. Thus, it is challenging to construct accurate and computationally efficient numerical schemes. To capture all scale information of the solution, a small mesh size is usually needed for spatial discretization to ensure accuracy. The stiffness arising from the multiscale coefficients requires using very small time steps to ensure stability for standard explicit temporal discretization. When the order of the time fractional derivative reaches zero, some numerical schemes turn out to be unstable due to the singularity \cite{stynes2017error, liao2018sharp}. The nonlocality of the fractional derivative also makes long-time simulation prohibitive, and some fast computational methods were proposed to address this issue \cite{jiang2017fast, gu2020parallel}. 

In previous works, many developed multiscale models have been proposed. These multiscale approaches include Multiscale Finite Element methods (MsFEM) \cite{efendiev2009multiscale,hou1997multiscale,hou1999convergence}, mixed MsFEM \cite{chen2003mixed}, numerical homogenized methods \cite{dixon1985order, lubich1982runge, bourgeat1984homogenized}, numerical upscaling methods \cite{durlofsky1991numerical, chen2003coupled, wu2002analysis}, generalized multiscale finite element methods (GMsFEM) \cite{efendiev2013generalized, efendiev2011multiscale, chung2015residual}, localized orthogonal decomposition (LOD) \cite{maalqvist2014localization}, constraint energy minimizing GMsFEM (CEM-GMsFEM) \cite{chung2018constraint} and nonlocal multicontinuum method (NLMC) \cite{chung2018non}, to name a few. Recently, the multicontinuum model is proposed based on the framework of multicontinuum homogenization \cite{efendiev2023multicontinuum, leung2024some}. It can capture the property of the solution with multiple quantities. The main idea is to represent the solution via averages and its gradient effects in each continua. For diffusion equations, some previous works have been developed \cite{contreras2023exponential, jin2019numerical,li2024wavelet}. Researchers have developed many methods to discretize the time fractional derivative. And they are based on the previous research about the existence and uniqueness of the solution of time fractional equations \cite{li2010existence,kemppainen2011existence}. The L1 scheme is one of the well known time stepping schemes that uses piecewise linear approximation on each time step \cite{lin2007finite, li2019constraint, podlubny1998fractional, yang2024accurate}. To overcome the difficulties caused by the stiffness and semilinear term, there has been a renewed interest in using exponential integrator (EI) methods \cite{cox2002exponential,garrappa2011accurate,garrappa2013exponential,garrappa2011accurate,hosseini2017solution,contreras2023exponential,li2023exponential,li2021high}. 

In this work, we propose a robust and accurate scheme to solve the semilinear high-contrast time fractional diffusion equation. We derive the semi-discretized coarse scale model based on the multicontinuum homogenization method. Then we adopt the exponential integrator method for temporal discretization to ensure stability and efficiency for the semilinear problem. EI can handle semilinear problems without nonlinear iterative methods, which can greatly improve computational efficiency for fractional problems. The convergence of our proposed approach is presented. The main contributions of this work are the following.
\begin{compactitem}
    \item We derive time fractional multicontinuum upscaled model and give the error estimate for the semi-discretization scheme.
    \item We then adopt one-step and two-step exponential integrator for temporal discretization for the upscaled model. The convergence results for the fully discretized system of the semilinear problem is presented.
    \item We perform some numerical experiments to show the accuracy and stability of our methods. We observe that the proposed method is more stable than classic method L1 scheme when the fractional order is small for the high-contrast subdiffusion equation.
\end{compactitem}


The paper is organized as follows. In Section \ref{Problem_setup}, we will introduce the problem setting and some preliminary definitions for the latter discussion. In Section \ref{Space_discretization}, we will discuss the spatial discretization and present the analysis of semi-discretized problems. We discuss the temporal discretization and present the analysis of full-discretized scheme in Section \ref{Time_discretization}. We present numerical results in Section \ref{Numerical_experiments} and give concluding remarks in the final section.

In order to simplify the notations, the notations \(C\), with or without a subscript, denote generic constants, which may differ at different occurrences, but it is always independent of any parameters.

\section{Problem setup}\label{Problem_setup}

Let  \(\Omega \subset \mathbb{R}^{2}\) be a bounded domain. In this paper, we focus on the following time fractional diffusion equation
\begin{equation}\label{e1}
    \left\{
        \begin{aligned}
            ^{C}_{0}{D}_{t}^{\alpha} u - \mathrm{div} (\kappa  \nabla u)  &= f(u) \ &\mathrm{in}\,&\Omega \times \bigl(0,T\bigr] ,\\
            u(t,x) &= 0 \ &\mathrm{on}\,&\partial \Omega \times \bigl(0,T\bigr] ,\\
            u(0,x) &= u_0(x) \ &\mathrm{in}\, &\Omega ,
        \end{aligned}
    \right.
\end{equation}
where \( ^{C}_{0}{D}_{t}^{\alpha} u \) is defined as the Caputo fractional derivatives of order \( \alpha \in (0,1) \) \cite{podlubny1998fractional}
\begin{equation}
    ^{C}_{0}{D}_{t}^{\alpha} u(x,t) = \frac{1}{\Gamma ( 1 - \alpha)} \int_{0}^{t} \frac{\partial u(x,s)}{\partial s} \frac{1}{ (t-s)^{\alpha}} \,\mathrm{d}s .
\end{equation}
We assume that the source term \(f(u)\) satisfies Lipschitz condition with constant \(L\) such that the problem \eqref{e1} is well-posed. We also assume that \(\kappa \) is a high-contrast heterogeneous permeability field. It is uniformly positive and bounded in \(\Omega, \ \ie \  \exists\,0< \kappa _{0} < \kappa _{1} < \infty \) such that \( \kappa _{0} \leq \kappa  \leq \kappa _{1}, \ \mathrm{in}\,\overline{\Omega} \).


There are three distinct scales which is highly related to our proposed homogenization method:
\begin{compactitem}
    \item The physical microscopic scale \(\varepsilon\), which represents the behavior and microscopic structure with high-contrast;
    \item The computational upscaling scale \(H\), where the homogenized problem is posed;
    \item The observing scale \(h\), where the multicontinuum quantities, such as local averages of the multiscale solution, are defined.
\end{compactitem}
These scales parameters satisfy \(\varepsilon \leq h < H\).

We start with the weak form of \eqref{e1}: find \(u \in H_{0}^{1} (\Omega) \) satisfying
\begin{equation}\label{weak_form}
    \bigl(\,_{0}^{C} D _{t}^{\alpha} u , v\,\bigr) + a_{\varepsilon} (u,v) = (f(u) , v) \quad  \forall\,v \in H_{0}^{1} (\Omega),
\end{equation}
where \((\cdot,\cdot)\) represent the standard \(L^{2}\) inner product and \(a_{\varepsilon}(\cdot,\cdot)\) is defined as \(a_{\varepsilon}(u,v) = \int_{\Omega} \kappa \nabla u \cdot \nabla v\). Let \(\omega _{i}\) be disjoint and \(\Omega = \bigcup\limits_{i} \omega _{i}\). We then extent the definition of the bilinear operator \(a_{\varepsilon}(\cdot,\cdot) \) to \(H^{1}(\Omega)\) as follows
\begin{equation*}
    a_{\varepsilon} (u,v) = \sum\limits_{i} \int_{\omega _{i}} \kappa  \nabla u | _{\omega _{i}} \cdot \nabla v | _{\omega _{i}}.
\end{equation*}
Moreover, we define \( \bigl\lVert u \bigr\rVert _{a(\omega)} \coloneqq \bigl(\,\int_{\omega} \kappa  \bigl\lvert\,\nabla u\,\bigr\rvert ^{2}\,\bigr)^{\frac{1}{2}} \) and denote \(\bigl\lVert \cdot \bigr\rVert _{a} = \bigl\lVert \cdot \bigr\rVert _{a(\Omega)} \). Let \(\bigl\lVert \cdot \bigr\rVert \) be the standard \(L^{2}\) norm of the domain \(\Omega\). We also assume the domain \(\Omega\) can be divided into two subregions \(\Omega = \Omega _{0} \cup \Omega _{1}\). The subdomains \(\Omega _{0}\) and \(\Omega _{1} \) represent the region with high and low conductivity respectively.

\section{Time fractional multicontinuum upscaling model}\label{Space_discretization}

In this section, we would like to derive the multicontinuum upscaling model. We start with the construction of local basis and present the related convergence analysis. The computational domain \(\Omega\) can be partitioned with respect to different scales. Given \(z \in [0,H] ^{d}\) ($d=2,3$), we define a rectangular partition \(\mathcal{T} _{H}(z)\) with mesh size \(H\)
\begin{equation*}
    \mathcal{T}_{H} (z) = \bigl\{\,K \subset \Omega \mid K = x_{i} + ( - \frac{H}{2} , \frac{H}{2} ) ^{d} \cap \Omega , \frac{x_{i} - z}{H} \in \mathbb{Z}^{d}\,\bigr\}, \ i=1,2,\cdots
\end{equation*}
For each coarse block \(K_{H}(x_{i}) = K_{z,H} (x_{i}) \coloneqq x_{i} + ( - \frac{H}{2} , \frac{H}{2} ) ^{d} \cap \Omega \in \mathcal{T}_{H}(z)\), we can separate the element into two subregions, \(K_{H,0}(x_{i}) = K_{H}(x_{i}) \cap \Omega _{0}, \ K_{H,1}(x_{i}) = K_{H}(x_{i}) \cap \Omega _{1}\). The index set is defined as \( I_{H} = I_{z,H} = \bigl\{\,x_{i} \mid K_{H}(x_{i}) \in \mathcal{T}_{H}(z) \,\bigr\}\), it contains all of the center points of the coarse elements in the partition \(\mathcal{T}_{H} (z)\). Similarly, we can define the other partition \(\mathcal{T} _{h}\) and its index set \(I _{h}\) with the mesh size \(h\). An illustration of the mesh grid is shown in Figure \ref{mesh_grid}.

\begin{figure}[H]
    \centering
    \includegraphics[width=13cm]{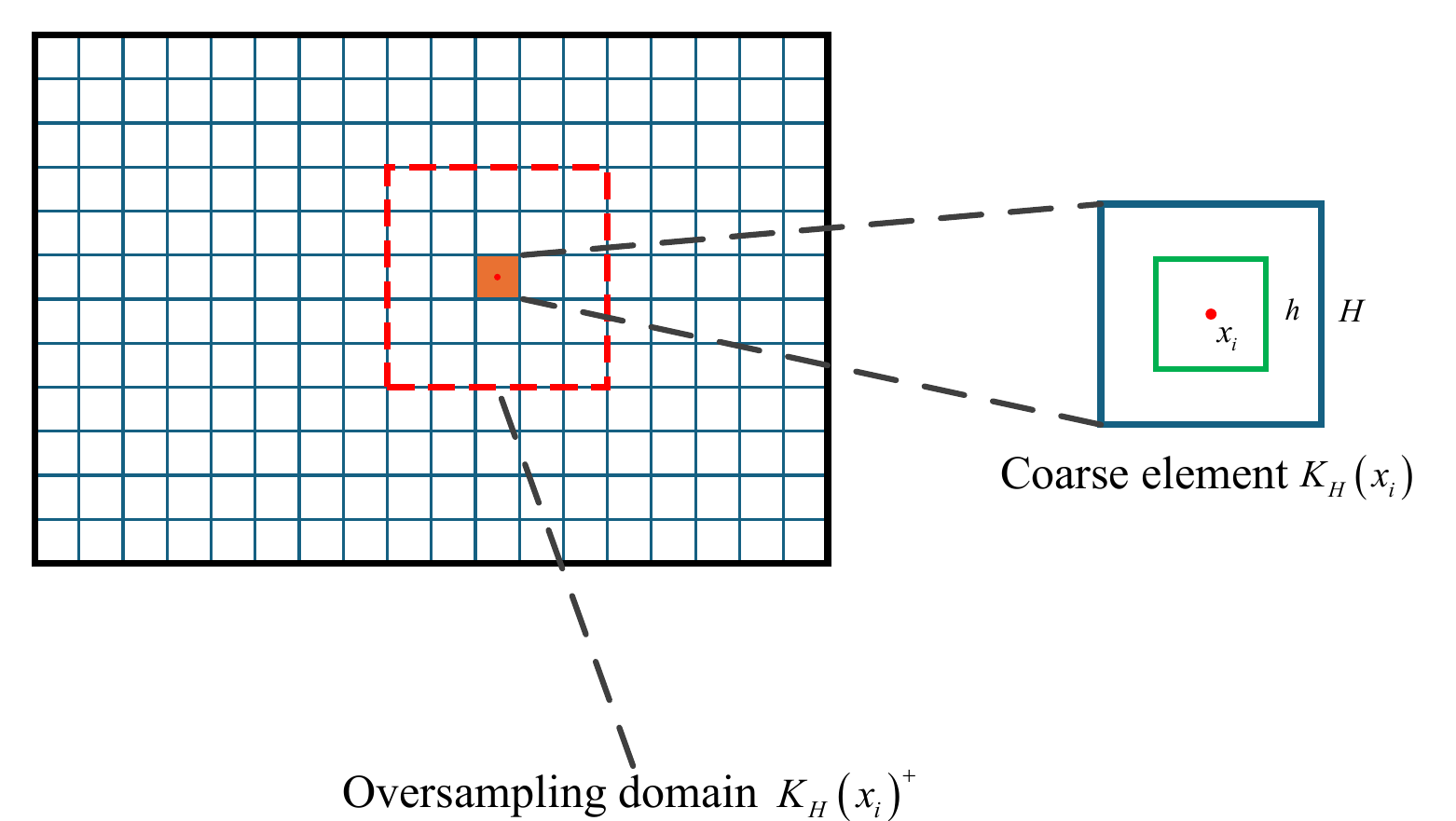}
    \caption{The fine grid, coarse grid \(K_{H}(x_{i})\), oversampling domain \(K_{H}(x_{i})^{+}\).}
    \label{mesh_grid}
\end{figure}

\subsection{Construction of local problems}

In this subsection, we construct the local problems for the multicontinuum upscaling model. We define the auxiliary basis \(\psi _{i,j} = \chi_{K _{H}(x_{i}) } = 1 \quad \mathrm{in \  continuum \ } i\) and \( \ 0 \quad \mathrm{otherwise}\). Based on CEM-GMsFEM/NLMC, the global basis functions \(\varphi _{l,j} \in H _{0}^{1} (\Omega) \quad j=0,1\) can be obtained by solving the constrained energy minimization problems. 
\begin{equation}
    \begin{aligned}
        a_{\varepsilon} (\varphi _{l,i},v) + \sum\limits_{x _{k} \in I_{z,h}}\sum\limits_{j=0,1} \eta ^{k,j} (\psi _{k,j}, v) &= 0,\\
        (\varphi_{l,i} , \psi _{k,j}) &= \delta_{ij} \delta _{lk} \int_{\Omega} \psi_{k,j},
    \end{aligned}
\end{equation}
where \(\eta ^{k,j}\) represents the energy constraint constant. We set the global multiscale space as \(V_{glo,h} = \mathop{\mathrm{span}} _{l,j} \bigl\{\,\varphi_{l,j}\,\bigr\} = \bigl\{\,v \in H_{0}^{1}(\Omega) \mid v = \sum _{l,j} v_{l,j} \varphi_{l,j}\,\bigr\} \). The convergence of the CEM-GMsFEM method with global basis is as follows \cite{chung2018non}, and will be used in our latter discussion.
\begin{lemma}\label{cem_global_basis_error_bound}
    Let \(f \in L^{2}(\Omega) \) and \( u_{glo,h} \in V_{glo,h} \) be the solution of \(a_{\varepsilon} (u_{glo,h} , v) = (f,v) \quad \forall\,v \in V_{glo,h}\). We have \(\int_{\Omega} (u_{glo,h} - u) \psi_{k,j} = 0 \quad \forall\,k,j\) and \( \bigl\lVert u_{glo,h} - u \bigr\rVert _{a} = O(h)\).
\end{lemma}
We can denote \(u_{glo,h} = \sum _{x_{l} \in I_{z}} U_{l,i} \, \varphi_{l,i}\) where the coefficients \(U_{l,i} = \dfrac{(u, \psi_{l,i})}{\int_{\Omega} \psi_{l,i}}\) are the local average of the solution \(u\) in the coarse elements \(K_{H}(x_{l})\).

In this work, we would like to introduce a time fractional multicontinuum model based on \cite{efendiev2023multicontinuum,leung2024some}, where the local problems correspond to the average and the gradient effect of the solution. To be more specific, to derive the homogenized equation on the coarse scale, we first need to find \( \phi_{i} \) and \( \phi_{i}^{m}\)  in \(i\)-th coarse region from the following two cell problems with constraints.
\begin{equation}\label{average_cellproblems}
    \left\{
        \begin{aligned}
            \int_{K_{H}^{+}(x)} \kappa  (y) \nabla _{y} \phi_{i}(x,y) \nabla _{y} v &=  \sum\limits_{k = 0,1} \sum\limits_{x_{j} \in I_{H}  \cap K_{H}(x) ^{+} } d_{jk} \bigl(\,\psi_{k}(x_{j},y),v\,\bigr), \\
            \bigl(\,\phi_{i} (x,y) , \psi _{k} (x_{j},y)\,\bigr) &= \delta_{ki} \int_{K_{H}(x)^{+}} \psi_{k} ( x_{j},y ) \quad \forall\,x_{j} \in I_{H} \cap K_{H}(x)^{+},
        \end{aligned}
    \right.
\end{equation}

\begin{equation}\label{gradienteffect_cellproblems}
    \left\{
        \begin{aligned}
            \int_{K_{H}^{+}(x)} \kappa (y) \nabla _{y} \phi_{i}^{m} (x,y) \nabla _{y} v &=  \sum\limits_{k = 0,1} \sum\limits_{x_{j} \in I_{H} \cap K_{H}(x) ^{+} } d_{jk} \bigl(\,\psi_{k}(x_{j},y),v\,\bigr), \\
            \bigl(\,\phi_{i}^{m} (x,y) , \psi _{k} (x_{j},y)\,\bigr) &= \delta_{ki} \int_{K_{H}(x)^{+}} (y^{m} - c_{x}^{(l)}) \psi_{k} ( x_{j},y ) \quad \forall\,x_{j} \in I_{H} \cap K_{H}(x)^{+},
        \end{aligned}
    \right.
\end{equation}
where \(K_{H}(x)^{+}\) is the oversampling domain which enrich \(K_{H} (x)\) by \(k \ h\)-layer, \(y^{(m)}\) is the \(l\)-th coordinate of \(y \in \mathbb{R}^{d}\) and \(c_{x}^{(m)}\) are some constants such that \(\int_{K_{H}(x)} (y^{(m)} - c_{x}^{(m)}) = 0\). The solution $\phi_{i}$ of the cell problem \eqref{average_cellproblems} describes the averages in each continua while the solution \( \phi_{i}^{m} \) of \eqref{gradienteffect_cellproblems} represents the linear basis function. The upscaling model with the following scaling holds
\begin{equation}\label{scalings}
    \bigl\lVert \phi_{i} \bigr\rVert = O(1), \ \bigl\lVert \nabla \phi_{i} \bigr\rVert = O(\frac{1}{\varepsilon}), \ \bigl\lVert \phi_{i}^{m} \bigr\rVert = O(\varepsilon), \ \bigl\lVert \nabla \phi_{i}^{m} \bigr\rVert = O(1).
\end{equation}

\subsection{Multi-continuum upscaling model}
Now, we present the construction of the multiscale upscaling model. We assume that in each coarse block \(K_{H}(x)\), the solution \(u\) could be approximated by \(u = \phi_{i} U_{i} + \phi_{i}^{m} \nabla_{m} U_{i}\), where \(U_{i}\) can be viewed as a limit of \(\frac{\int_{K_{H}(x)} u\psi_{i}}{\int_{K_{H}(x)} \psi_{i}}\) as the block size goes to zero. It represents the average solution in \(i\)-th continua. With this expression of \(u\), we can write
\begin{equation}\label{diffusion_term_approx}
    \int_{K_{H}(x)} \kappa  \nabla u \cdot \nabla v = \int_{K_{H}(x)} \kappa  \nabla (\phi_{i} U_{i}) \cdot \nabla v + \int_{K_{H}(x)} \kappa  \nabla (\phi_{i}^{m} \nabla_{m} U_{i}) \cdot \nabla v.
\end{equation}
Approximately, we have
\begin{equation*}
    \begin{aligned}
        \int_{K_{H}(x)} \kappa  \nabla (\phi_{i} U_{i}) \cdot \nabla v &= \int_{K_{H}(x)} \kappa  (\nabla \phi_{i}) U_{i} \cdot \nabla v \\
        &=  U_{i}(x) \int_{K_{H}(x)} \kappa  \nabla \phi_{i} \cdot \nabla v ,\\
        \int_{K_{H}(x)} \kappa  \nabla(\phi_{i}^{m} \nabla_{m} U_{i}) \cdot \nabla v &= \int_{K_{H}(x)} \kappa  \nabla (\phi_{i}^{m}) \nabla_{m} U_{i} \cdot \nabla v \\
        &= \nabla_{m} U_{i}(x) \int_{K_{H}(x)} \kappa  \nabla \phi_{i}^{m} \cdot \nabla v.
    \end{aligned}
\end{equation*}
We take \(v = \phi _{j} V_{j} + \phi _{j}^{n} \nabla _{n} V _{j}\). Substituting the above equations into \eqref{diffusion_term_approx}, we obtain
\begin{equation*}
    \begin{aligned}
        &\int_{K_{H}(x)} \kappa  \nabla u \cdot \nabla v = \int_{K_{H}(x)} \kappa  \nabla (\phi_{i} U_{i}) \cdot \nabla v + \int_{K_{H}(x)} \kappa  \nabla (\phi_{i}^{m} \nabla_{m} U_{i}) \cdot \nabla v\\
        = &\nabla _{m} U_{i}(x) \nabla _{n} V_{j}(x) \int _{K_{H}(x)} \kappa \nabla \phi _{i}^{m} \cdot \nabla \phi _{j}^{n} + \nabla _{m} U_{i}(x) V_{j}(x) \int _{K_{H}(x)} \kappa \nabla \phi _{i}^{m} \cdot \nabla \phi _{j}\\
        &+ U_{i}(x) \nabla _{n} V_{j}(x) \int _{K_{H}(x)} \kappa \nabla \phi _{i} \cdot \nabla \phi _{j}^{n} +  U_{i}(x) V_{j}(x) \int _{K_{H}(x)} \kappa \nabla \phi _{i} \cdot \nabla \phi _{j} .
    \end{aligned}
\end{equation*}
Denote by \(\alpha_{ij}^{mn} = \int_{K_{H}(x)} \kappa \nabla \phi_{i}^{m} \cdot \nabla \phi_{j}^{n}, \ \beta_{ij}^{m} = \int_{K_{H}(x)} \kappa \nabla \phi_{i}^{m} \cdot \nabla \phi_{j}, \ \gamma_{ij} = \int_{K_{H}(x)} \kappa \nabla \phi_{i} \cdot \nabla \phi_{j}\). The above equations can be simplified as
\begin{equation*}
    \int_{K_{H}(x)} \kappa  \nabla u \cdot \nabla v = \nabla_{m} U_{i} \alpha_{ij}^{mn} \nabla_{n} V_{j} + \nabla_{m} U_{i} \beta_{ij}^{m} V_{j} + U_{i} \beta_{ji}^{n} \nabla_{n} V_{j} + U_{i} \gamma_{ij} V_{j}.
\end{equation*}
Assume the cell problems do not depend on time \(t\), for the time fractional problem \eqref{weak_form}, we have \(^{C}_{0}{D}_{t}^{\alpha} u = \phi_{i} \,_{0}^{C}D_{t}^{\alpha} U_{i} + \phi_{i}^{m} \,_{0}^{C}D_{t}^{\alpha} \nabla_{m}U_{i}\)
\begin{equation*}
    \begin{aligned}
        \int_{K_{H}(x)} \,_{0}^{C}D_{t}^{\alpha} u \, v = &\,_{0}^{C}D_{t}^{\alpha} U_{i} V_{j} \int_{K_{H}(x)} \phi_{i} \phi_{j} +  \,_{0}^{C}D_{t}^{\alpha} \nabla_{m}U_{i} V_{j} \int_{K_{H}(x)} \phi_{i}^{m} \phi_{j}\\
        &+ \,_{0}^{C}D_{t}^{\alpha} U_{i} \nabla_{n} V_{j} \int_{K_{H}(x)} \phi_{i} \phi_{j}^{n} + \,_{0}^{C}D_{t}^{\alpha} \nabla_{m}U_{i} \nabla_{n} V_{j} \int_{K_{H}(x)} \phi_{i}^{m} \phi_{j}^{n}.\\       
    \end{aligned}
\end{equation*}
We neglect the higher order terms based on scalings \eqref{scalings}. Let \(m_{ij} = \int_{K_{H}(x)} \phi_{i} \phi_{j}\), we derive the time fractional derivative term as
\begin{equation*}
    \int_{K_{H}(x)} \,_{0}^{C}D_{t}^{\alpha} u  v = \, _{0}^{C}D_{t}^{\alpha} U_{i} m_{ij} V_{j} .
\end{equation*}

We then define rescaled quantities as follows
\begin{equation*}
    \widehat{\alpha_{ij}^{mn}} = \frac{1}{\bigl\lVert K_{H}(x) \bigr\rVert} \alpha_{ij}^{mn}, \ \widehat{\beta_{ij}^{m}} = \frac{\varepsilon}{\bigl\lVert K_{H}(x) \bigr\rVert} \beta_{ij}^{m}, \ \widehat{\gamma_{ij}} = \frac{\varepsilon^{2}}{\bigl\lVert K_{H}(x) \bigr\rVert} \gamma_{ij}, \  \widehat{m _{ij}} = \frac{1}{\bigl\lVert K_{H}(x) \bigr\rVert} m_{ij}.
\end{equation*}
Thus, the weak form of the time fractional homogenized problem can be written as

\begin{equation}\label{weak_homo_form_time}
    \begin{aligned}
        &\int_{\Omega} \widehat{m_{ij}} \,_{0}^{C}D_{t}^{\alpha} U_{i} V_{j} \\
        + &\int_{\Omega} \widehat{\alpha_{ij}^{mn}} \nabla_{m} U_{i} \nabla_{n} V_{j} + \frac{1}{\varepsilon} \int_{\Omega} \widehat{\beta_{ij}^{m}} \nabla_{m} U_{i} V_{j} + \frac{1}{\varepsilon} \int_{\Omega} \widehat{\beta_{ji}^{n}} \nabla_{n} V_{j} U_{i} + \frac{1}{\varepsilon^{2}} \int_{\Omega} \widehat{\gamma_{ij}} U_{i} V_{j}\\
        = &\int_{\Omega} f(U_{i}) V_{j}.
    \end{aligned}
\end{equation}
Let
\begin{equation*}
    \begin{aligned}
        M_{ij} &= \int_{\Omega} \widehat{m_{ij}} \phi_{i} \phi_{j}, \\
        A_{ij} &= \int_{\Omega} \widehat{\alpha _{ij}^{mn}} \nabla _{m} \phi_{i}^{m} \nabla _{n} \phi_{j}^{n} + \frac{1}{\varepsilon} \widehat{\beta _{ij}^{m}} \nabla _{m} \phi_{i}^{m} \phi_{j} + \frac{1}{\varepsilon} \widehat{\beta _{ji}^{n}} \nabla _{n} \phi_{j}^{n}  \phi_{i} + \frac{1}{\varepsilon ^{2}} \widehat{\gamma _{ij}} \phi_{i} \phi_{j},
    \end{aligned}
\end{equation*}
The matrix form of \eqref{weak_homo_form_time} is
\begin{equation}\label{matrix_form}
    M \, _{0}^{C} D_{t}^{\alpha} U(t) + AU(t) = f(U(t)),
\end{equation}
where \(M = (M_{ij}), \ A = (A_{ij})\), which are symmetric and positive definite.

\subsection{Convergence of multi-continuum upscaling model}


In this subsection, we present the convergence of the semi-discretized problem. Firstly, we define the nonlocal multicontinuum (NLMC) downscaling operators.
\begin{equation}\label{nlmc_downscaling_operator}
    \begin{aligned}
        P_{h,H} \colon \bigl( C^{1}(\Omega) \bigr) ^{2} &\to H^{1}(\mathcal{T}_{H} (z)),\\
        (U_{0},U_{1}) &\mapsto \sum\limits_{i=0,1} \sum\limits_{x_{l} \in I_{H}} \chi_{K_{H}(x_{l})} \bigl( \phi_{i} U_{i}(x_{l}) + \phi_{i}^{m} \nabla_{m} U_{i}(x_{l}) \bigr),
    \end{aligned}
\end{equation}
where \( H^{1}(\mathcal{T}_{H} (z)) \coloneqq \bigl\{ \, u \in L^{2}(\Omega) \mid u| _{K} \in H^{1}(K), \ \forall\,K \subset \Omega \,\bigr\} \). We also define another downscaling operator on the space formed by CEM-GMsFEMs' global basis as follows.
\begin{equation}\label{cem_global_downscaling_operator}
    \begin{aligned}
        P_{h} \colon \bigl( C^{1}(\Omega) \bigr) ^{2} &\to H^{1}(\Omega),\\
        (U_{0},U_{1}) &\mapsto \sum\limits_{i=0,1} \sum\limits_{x_{l} \in I_{h}} U_{i}(x_{l}) \varphi_{l,i}^{z}.
    \end{aligned}
\end{equation}

Define the \(L^{2}\) projection operator \(\Pi _{i,h}\) on the auxiliary space \(V_{\mathrm{aux},i} = \mathop{\mathrm{span}} _{l} \bigl\{ \psi_{l,i} \bigr\}\) as
\begin{equation}
    \Pi _{i,h} (v) = \sum\limits_{x_{i} \in I_{z,h}} \frac{\int_{K_{z,h}(x_{l}) \cap \Omega_{i}} v}{\int_{K_{z,h}(x_{l}) \cap \Omega_{i}} 1} \psi_{l,i}^{z} \ i=0,1 .
\end{equation}
The \(L^{2}\) projection of the exact solution \(u\) might lack smoothness, we assume that \(\Pi _{i,h} (u)\) can be extended to a function that belongs to \(C^{1} (\Omega)\) with \(h \to 0\), and use the notation \(\Pi _{i,h} (u)\) to represent a smooth function thereafter. Then, we have \(P_{h} \bigl( \Pi _{0,h} u, \Pi _{1,h} u \bigr) = u_{glo,h}\). Notice that the local average of the summation on the domain \(\Omega\) with respect to the point \(z\) can be represented as follows: for any \(g \in C ^{1} (\Omega)\), we have
\begin{equation*}\label{local_average}
    \int_{[0,H]^{d}} \sum\limits_{x_{l} \in I_{z,H}} g(x_{l})  = \int_{\Omega} g(x)\,\mathrm{d}x.
\end{equation*}

Then a bilinear operator \(\widetilde{a}_{glo,h} \colon \bigl( C^{1}(\Omega) \bigr) ^{2} \times \bigl( C^{1}(\Omega) \bigr) ^{2} \to \mathbb{R}\) can be defined by
\begin{equation*}
    \widetilde{a}_{glo,h} ((U_{0},U_{1}),(V_{0},V_{1})) \coloneqq a_{\varepsilon} \bigl( P_{h} ( U_{0},U_{1} ), P_{h}(V_{0},V_{1}) \bigr).
\end{equation*}
By Lemma \refeq{cem_global_basis_error_bound}, we obtain \(\forall\,v \in H_{0}^{1}(\Omega)\),
\begin{equation}
        \widetilde{a}_{glo,h} \bigl( ( \Pi _{0,h} u, \Pi _{1,h} u ),  ( \Pi _{0,h} v,\Pi _{1,h} v ) \bigr) = \bigl( ( \Pi _{0,h} f(u),\Pi _{1,h} f(u) ), ( \Pi _{0,h} v,\Pi _{1,h} v ) \bigr).
\end{equation}
Furthermore, a homogenized bilinear operator \(\widetilde{a}_{h,H} \colon \bigl( H^{1}(\Omega) \bigr) ^{2} \times \bigl( H^{1}(\Omega) \bigr) ^{2} \to \mathbb{R} \) can be defined as
\begin{equation}
    \begin{aligned}
        &\widetilde{a}_{h,H} ((U_{0},U_{1}),( V_{0},V_{1})) \\
        \coloneqq &\int_{\Omega} \frac{1}{H^{d}} \int_{K_{H}(x)} \kappa  (y) \nabla _{y} \sum\limits_{i} \bigl(U_{i}(x) \phi_{i}(x,y) + \nabla _{m} U_{i}(x) \phi_{i}^{m}(x,y) \bigr) \cdot \\
        &\sum\limits_{j}\nabla _{y} \bigl(V_{i}(x) \phi_{i}(x,y) + \nabla _{m} V_{i}(x) \phi_{i}^{m}(x,y) \bigr) \, \mathrm{d}y \,\mathrm{d}x\\
        = &\int_{\Omega} \sum\limits_{i,j} \bigl( \widehat{\alpha_{ij}^{mn}} \nabla _{m} U_{i} \nabla _{n} V_{j} + \widehat{\beta_{ij}^{m}} \nabla _{m} U_{i} V_{j} + \widehat{\beta_{ij}^{n}} U_{i} \nabla _{n}V_{j} + \widehat{\gamma_{ij}} U_{i} V_{j} \bigr).
    \end{aligned}
\end{equation}
The following Lemma describes the property of the homogenized operator \cite{leung2024some}.
\begin{lemma}
    The operator \(\widetilde{a}_{h,H} \) is an averaging of the multiscsale operator \(a_{\varepsilon}\) in sense of 
    \begin{equation}
        \widetilde{a} _{h,H} ((U_{0},U_{1}),( V_{0},V_{1})) = \frac{1}{H^{d}} \int_{[0,H]^{d}} a_{\varepsilon} (P_{h,H} (U_{0},U_{1}), P_{h,H} (V_{0},V_{1}) ) . 
    \end{equation}
\end{lemma}

We consider \(  \bigl( U_{0,h,H}(t),U_{1,h,H}(t) \bigr) \in \bigl( H^{1}(\Omega) \bigr)^{2} \) to be the solution of the following upscaled problem satisfying
\begin{equation}\label{upscaled_problems}
    \begin{aligned}
        &\frac{1}{H^{d}} \int_{[0,H]^{d}} \biggl( \,_{0}^{C} D_{t}^{\alpha} P_{h,H}( U_{0,h,H}(t) , U_{1,h,H}(t) ) , P_{h,H}( V_{0},V_{1} ) \biggr) \\
        + &\widetilde{a}_{h,H} \bigl( (U_{0,h,H}(t),U_{1,h,H}(t)), (V_{0},V_{1}) \bigr) \\
        = &\frac{1}{H^{d}} \int_{[0,H]^{d}} \biggl( P_{h,H} \bigl(  f( U_{0,h,H}(t)), f ( U_{1,h,H}(t)) \bigr), P_{h,H} (V_{0},V_{1}) \biggr) \quad \forall\,(V_{0},V_{1}) \in \bigl(H^{1}(\Omega)\bigr)^{2}.
    \end{aligned}
\end{equation}
The following Lemma demonstrates the approximation of the two downscaling operators with the scale \(h \ll H\) if \(U_{0}, U_{1}\) is smooth sufficiently. We notice that the global basis \(\varphi_{l,i}^{z}\) shows the exponential decaying property \cite{chung2018constraint}.
\begin{lemma}\label{downscaling_operator_approximation}
    Given \(U_{0}, U_{1} \in H^{1+s}(\Omega)\) with \(0 < s\leq 1\). If the number of the oversampling layer \(k\) satisfies \(2kh < H  = O(kh)\).  For some constant \(C_{0} \geq k  = O(\log(H^{-1}))\), we have \cite{leung2024some}
    \begin{equation}
        \bigl\lVert P_{h,H}(U_{0}(t),U_{1}(t)) - P_{h}(U_{0}(t),U_{1}(t) ) \bigr\rVert  \leq C \log (\frac{1}{H}) H^{1+s} \bigl\lVert (U_{0}(t),U_{1}(t)) \bigr\rVert.
    \end{equation}
\end{lemma}

The relationship between the solution \(u\) and the \(L^{2}\) projection \( P_{h}(\Pi _{i,h} u)\) can be described as
\begin{equation}\label{a_hH_Galerkin}
    a_{\varepsilon} \bigl( u(t) - \frac{1}{H^{d}} \int_{ [0,H]^{d}} P_{h} (\Pi _{0,h} u(t), \Pi _{1,h} u(t) ) , v \bigr) = 0  \quad  \forall\,v \in V_{glo,h}.
\end{equation}

Moreover, for the time fractional derivative, we have the following Lemma.
\begin{lemma}\label{time_fractional_lemma}
    For any function \(v(t)\) absolutely continuous on \([0,T]\), we have \cite{yang2024accurate}
    \begin{equation}
        v(t) \, _{0}^{C} D _{t}^{\alpha} v(t) \geq \frac{1}{2} \, \,_{0}^{C} D _{t}^{\alpha} v^{2}(t) \quad  0 < \alpha < 1.
    \end{equation}
\end{lemma}

With the approximation of the two downscaling operators and by the assumption of the smoothness of \(U_{0,h,H}\) and \(U_{1,h,H}\), we derive the convergence of semi-discretize problem as following theorem based on above lemmas.
\begin{theorem}\label{space_approximation}
    Let \(u\) be the solution to \eqref{e1}. With the same condition in Lemma \refeq{downscaling_operator_approximation}, we have following estimates
    \begin{equation}
         \max\limits_{0\leq t\leq T}   \bigl\lVert u(t) - \frac{1}{H^{d}} \int_{ [0,H]^{d} }  P_{h,H} \bigl(  U_{0,h,H}(t), U_{1,h,H}(t) \bigr) \bigr\rVert \leq C h^{2} .
    \end{equation}
\end{theorem}

\begin{proof}
    We estimate the error by three parts.
    \begin{equation*}
        \begin{aligned}
            &\bigl\lVert u(t) - \frac{1}{H^{d}} \int_{[0,H] ^{d}} P_{h,H} \bigl(  U_{0,h,H}(t) , U_{1,h,H}(t)  \bigr) \bigr\rVert\\
            \leq &\bigl\lVert u(t) - \frac{1}{H^{d}} \int_{[0,H] ^{d}} P_{h} ( \Pi _{0,h} u(t) , \Pi _{1,h} u(t)  )\bigr\rVert\\
            &+ \bigl\lVert \frac{1}{H^{d}} \int_{[0,H] ^{d}} ( P_{h,H} - P_{h} ) ( \Pi _{0,h} u(t) , \Pi _{1,h} u(t)  ) \bigr\rVert\\
            &+ \bigl\lVert \frac{1}{H^{d}} \int_{[0,H] ^{d}} P_{h,H} \bigl( ( \Pi _{0,h} u(t) , \Pi _{1,h} u(t) ) - ( U_{0,h,H}(t) , U_{1,h,H}(t) ) \bigr) \bigr\rVert\\\
            = &\bigl\lVert e_{1}(t) \bigr\rVert + \bigl\lVert e_{2}(t) \bigr\rVert + \bigl\lVert e_{3}(t) \bigr\rVert,
        \end{aligned}
    \end{equation*}
    where \(e_{1}(t)\) represents the error between the exact solution and its projection on global CEM-GMsFEM space, \(e_{2}\) measures the difference between two downscaling operators. We have already obtained these estimates in Lemma \refeq{cem_global_basis_error_bound}. We will then focus on estimating the remaining error \(e_{3}(t)\). By upscaled problems \eqref{upscaled_problems}, we have
    \begin{equation*}
        \begin{aligned}
            &\frac{1}{H^{d}} \int_{[0,H]^{d}} \bigl( \,_{0}^{C} D _{t}^{\alpha} e_{3}(t) , P_{h,H} (V_{0} , V_{1}) \bigr) + \frac{1}{H^{d}} \int_{[0,H]^{d}} a_{\varepsilon} \bigl( e_{3}(t) , P_{h,H} (V_{0} , V_{1}) \bigr)\\
            = &\frac{1}{H^{d}} \int_{[0,H]^{d}} \bigl( \,_{0}^{C} D _{t}^{\alpha} e_{2}(t) , P_{h,H} (V_{0} , V_{1}) \bigr) + \frac{1}{H^{d}} \int_{[0,H]^{d}} a_{\varepsilon} \bigl( e_{2}(t) , P_{h,H} (V_{0} , V_{1}) \bigr)\\
            &+ \frac{1}{H^{d}} \int_{[0,H]^{d}} \bigl( \,_{0}^{C} D _{t}^{\alpha} (  P_{h}  ( \Pi _{0,h} u(t) , \Pi _{1,h} u(t) ) -  P_{h,H} (U_{0,h,H}(t) , U_{1,h,H}(t)) ) , P_{h,H} (V_{0} , V_{1}) \bigr)\\
            &+ \frac{1}{H^{d}} \int_{[0,H]^{d}} a_{\varepsilon} \bigl( ( P_{h} ( \Pi _{0,h} u(t) ,\Pi _{1,h} u(t)  ) - P_{h,H} (U_{0,h,H}(t) , U_{1,h,H}(t) ) ) , P_{h,H} (V_{0} , V_{1}) \bigr) .\\
        \end{aligned}
    \end{equation*}
    According to the definition of downscaling operators \eqref{nlmc_downscaling_operator}, the weak form \eqref{weak_form} and \eqref{a_hH_Galerkin}, we have
    \begin{equation*}
        \begin{aligned}
            &\frac{1}{H^{d}} \int_{[0,H]^{d}} \bigl( \,_{0}^{C} D _{t}^{\alpha} (  P_{h}  ( \Pi _{0,h} u(t) , \Pi _{1,h} u(t) ) -  P_{h,H} (U_{0,h,H}(t) , U_{1,h,H}(t) ) ) , P_{h,H} (V_{0} , V_{1}) \bigr)\\
            &+ \frac{1}{H^{d}} \int_{[0,H]^{d}} a_{\varepsilon} \bigl( ( P_{h} (  \Pi _{0,h} u(t) ,\Pi _{1,h} u(t)  )- P_{h,H} (U_{0,h,H}(t) , U_{1,h,H}(t) ) ) , P_{h,H} (V_{0} , V_{1}) \bigr)\\
            = &\frac{1}{H^{d}} \int_{[0,H]^{d}} \bigl(  \,_{0}^{C} D _{t}^{\alpha} P_{h}  ( \Pi _{0,h} u(t) , \Pi _{1,h} u(t) ) , ( P_{h,H} - P_{h}) (V_{0},V_{1}) \bigr)\\
        \end{aligned}
    \end{equation*}
    \begin{equation*}
        \begin{aligned}
            &+ \frac{1}{H^{d}} \int_{[0,H]^{d}} a_{\varepsilon} \bigl( P_{h}  ( \Pi _{0,h} u(t) , \Pi _{1,h} u(t) ) ,P_{h,H} (V_{0},V_{1}) \bigr) - a_{\varepsilon} ( u(t),\frac{1}{H^{d}} \int_{[0,H]^{d}} P_{h}(V_{0},V_{1}) )\\
            &+ \frac{1}{H^{d}} \int_{[0,H]^{d}} \bigl( f(P_{h}(\Pi _{0,h} u(t) , \Pi _{1,h} u(t) ) ) , P_{h}(V_{0},V_{1})  \bigr) \\
            &- \frac{1}{H^{d}} \int_{[0,H]^{d}} \bigl( f(P_{h,H}(\Pi _{0,h} u(t) , \Pi _{1,h} u(t) ) ) , P_{h,H}(V_{0},V_{1})  \bigr)\\
        \end{aligned}
    \end{equation*}
    Take \((V_{0} , V_{1}) = ( \Pi _{0,h} u(t) , \Pi _{1,h} u(t) ) - ( U_{0,h,H}(t) , U_{1,h,H}(t) )\), we obtain that
    \begin{equation*}
        \begin{aligned}
            &- \frac{1}{H^{d}} \int_{[0,H]^{d}} \bigl(   f(u(t)) -  f(P_{h}(\Pi _{0,h} u(t) , \Pi _{1,h} u(t)) ) , e_{3}(t)  \bigr)\\
            = &- \bigl( \,_{0}^{C} D _{t}^{\alpha} e_{1}(t) , e_{3}(t) \bigr) \\
            &+ \bigl(  \,_{0}^{C} D _{t}^{\alpha} e_{1}(t) , \frac{1}{H^{d}} \int_{[0,H]^{d}} ( P_{h,H} - P_{h}) (( \Pi _{0,h} u(t) , \Pi _{1,h} u(t) ) - ( U_{0,h,H}(t) , U_{1,h,H}(t) )) \bigr)\\
            &+ \frac{1}{H^{d}} \int_{[0,H]^{d}} \bigl(   f(P_{h,H}(\Pi _{0,h} u(t) , \Pi _{1,h} u(t)) ) - f(P_{h,H}(U_{0,h,H}(t) , U_{1,h,H}(t)) ) , e_{3}(t)   \bigr)\\
            &- \frac{1}{H^{d}} \int_{[0,H]^{d}} \biggl(   f(u(t)) -  f(P_{h}(\Pi _{0,h} u(t) , \Pi _{1,h} u(t)) ) , \\
            &\ ( P_{h,H} - P_{h}) (( \Pi _{0,h} u(t) , \Pi _{1,h} u(t) ) - ( U_{0,h,H}(t) , U_{1,h,H}(t) ))  \biggr)\\
        \end{aligned}
    \end{equation*}
    By Young's inequality and Lemma \refeq{downscaling_operator_approximation}
    \begin{equation*}
        \begin{aligned}
            &\frac{1}{2} \,_{0}^{C} D _{t}^{\alpha} \bigl\lVert e_{3}(t) \bigr\rVert ^{2} + \bigl\lVert e_{3}(t) \bigr\rVert _{a}^{2}\\
            \leq & L \bigl\lVert e_{1}(t) \bigr\rVert \bigl\lVert e_{3}(t) \bigr\rVert + L \bigl\lVert e_{3}(t) \bigr\rVert ^{2} + CL \log (\frac{1}{H}) H^{1+s} \bigl\lVert e_{1}(t) \bigr\rVert \cdot \bigl\lVert e_{3}(t) \bigr\rVert\\
            \leq & \bigl\lVert e_{3}(t) \bigr\rVert ^{2} + C \bigl\lVert e_{1}(t) \bigr\rVert ^{2} + C \bigl( \log (\frac{1}{H}) \bigr)^{2} H^{2+2s} \bigl\lVert e_{1}(t) \bigr\rVert ^{2},
        \end{aligned}
    \end{equation*}
    We apply the Poincar\'e inequality and the fractional Gr\"onwall's inequality \cite{jin2018numerical}, then conduct the Riemann-Liouville fractional integration and obtain
    \begin{equation*}
        \max\limits_{0\leq t\leq T}\bigl\lVert e_{3}(t) \bigr\rVert \leq C h^{2}.
    \end{equation*}
    This completes the proof.

\end{proof}

\section{Time discretization}\label{Time_discretization}

In this section, we focus on the approximation of the time fractional derivative \( ^{C}_{0}{D}_{t}^{\alpha} U(t) \). We divide the time interval \( [0,T] \) into \(N\) equidistant parts and denote \(t_{n} = n \tau, \ n = 0,1,\cdots,N \), where \(\tau = \frac{T}{N}\) is the time step size. One of the common way to approximate the time fractional derivative \(  ^{C}_{0}{D}_{t}^{\alpha} U(t)  \) is the L1 scheme \cite{lin2007finite}
\begin{equation}
    ^{C}_{0}{D}_{t}^{\alpha} U(t_{n+1}) = \frac{1}{\Gamma (2- \alpha )}  \sum\limits_{j=0}^{n} \frac{ U(t_{n+1-j}) - U(t_{n-j}) }{\tau ^{\alpha}} b_{j}, \ n = 0,1,\cdots,N ,
\end{equation}
where \(b_{j} = (j+1)^{1 - \alpha} - j^{1 - \alpha}, j=0,1,\cdots,N\). For all \(n = 0,1,\cdots,N \), the variational formulation of the \eqref{upscaled_problems} reads: find \(U^{n+1} = ( U_{0}^{n+1} , U_{1}^{n+1} ) \in \bigl( H^{1}(\Omega) \times H^{1}(\Omega) \bigr) ^{2}\), such that
\begin{equation}\label{fully_discrete_weak_form}
    \begin{aligned}
        (U^{n+1} , V)_{H} + \alpha _{0} \widetilde{a} _{h,H} (U^{n+1},V) = &(1-b_{1}) (U^{n} , V)_{H} + \sum\limits_{j=1}^{n-1}(b_{j} - b_{j+1}) (U^{n-j} , V)_{H} \\
        &+ b_{n} (U^{0} , V)_{H}  + \alpha _{0} (f(U^{n+1}) , V)_{H},
    \end{aligned}
\end{equation}
for all \(V = (V_{0},V_{1}) \in \bigl(H^{1}(\Omega)\bigr)^{2}\), where \(\alpha _{0} = \Gamma (2-\alpha) \tau^{\alpha}\) and \((\cdot,\cdot)_{H}\) is defined as
\begin{equation*}
    (U,V)_{H} \coloneqq \frac{1}{H^{d}} \int_{[0,H]^{d}} \bigl( (U_{0},U_{1}) , (V_{0},V_{1}) \bigr).
\end{equation*}
\(U^{n} \) is the approximation of \(U(t_{n})\). We could derive the matrix form \eqref{matrix_form} based on the above discrete variational formulation \eqref{fully_discrete_weak_form}
\begin{equation}\label{L1_fully_discretized}
    (M + \alpha_{0} A) U^{n+1} = \alpha_{0} f(U^{n+1}) + (1-b_{1})M U^{n} + M \sum\limits_{j=1}^{n-1} (b_{j} - b_{j+1}) U^{n-j} + b_{n} M U^{0}.
\end{equation}
The upscaled stiffness matrix \(A\) explains the homogenized bilinear operator \(\widetilde{a}_{h,H} (\cdot,\cdot) \) while the upscaled mass matrix \(M\) represents the homogenized inner product \((\cdot,\cdot)_{H}\).

\subsection{Exponential Integrator}

Instead of classic L1 scheme, we introduce the exponential integrator approach in this work. Exponential integrator method has shown to be stable for stiff problems \cite{garrappa2013exponential} and more efficient for semilinear problems. Thus is suitable for our semi-discretized equation \eqref{weak_homo_form_time}. We rewrite the semilinear problem \eqref{matrix_form} as \(^{C}_{0}{D}_{t}^{\alpha} U (t) + K U(t) = F(U(t))\) where \(K = M^{-1}A\) and \(F(U(t)) = M^{-1} f(t) \), and let \(\widehat{f}\) be the Laplace transform of a function \(f\) 
\begin{equation*}
    \widehat{f} (s) = \int_{0}^{\infty} f(t) \mathrm{e} ^{-st} \,\mathrm{d}t.
\end{equation*}

Performing the Laplace transform on the matrix equation, we have
\begin{equation*}
    \widehat{^{C}_{0}{D}_{t}^{\alpha} U} (s) + K \widehat{U}(s) = \widehat{F(U)}(s).
\end{equation*}
The Laplace transform of the Caputo fractional derivative can be written as \( \widehat{^{C}_{0}{D}_{t}^{\alpha} U } (s) = s^{\alpha} \widehat{U}(s) - s^{\alpha - 1} \widehat{U}(0) \) \cite{podlubny1998fractional}. Thus, we have
\begin{equation*}
    \widehat{U} (s) = (s^{\alpha} I + K )^{-1} \bigl( \widehat{F(U)}(s) + s^{\alpha - 1} \widehat{U} (0)  \bigr).
\end{equation*}
Utilizing the inverse Laplace transform, we obtain 
\begin{equation}\label{exponential_integrator_solution}
    \begin{aligned}
        U(t) &= e _{\alpha,1} (t,K) U(0) + \int_{0}^{t} e _{\alpha,\alpha} (t-r,K) F(U(r)) \,\mathrm{d}r\\
        &= e_{\alpha,1} (t,K) U (0) + \sum\limits_{j=0}^{n-1} \int_{t_{j}}^{t_{j+1}} e_{\alpha,\alpha} (t-r,K) F(U(r)) \,\mathrm{d}r ,
    \end{aligned}
\end{equation}
where the function \(e_{\alpha, \beta} (t,\lambda)\) denotes the generalization of the inverse of the Laplace transform \((s^{\alpha} + \lambda)^{-1} s^{\alpha - \beta}\) to the matrix argument \(K\). Note that \(e_{\alpha, \beta}(t,\lambda)\) can be computed by \(e_{\alpha,\beta} (t,\lambda) = t^{\beta -1} E_{\alpha,\beta} (- t^{\alpha} \lambda)\), where \(E_{\alpha,\beta}(z)\) is the Mittag-Leffler function \cite{podlubny1998fractional}
\begin{equation*}
    E_{\alpha, \beta} (z) = \sum\limits_{k=0}^{\infty} \frac{z^{k}}{\Gamma( \alpha k + \beta)}.
\end{equation*}

The following lemma illustrates the property of the function \(e_{\alpha,\beta}(t,\lambda)\) \cite{garrappa2013exponential}
\begin{lemma}
    Suppose that \( a < t \) and \(\alpha,\beta > 0\) and let \(r \in \mathbb{R}\) such that \(r>-1\), then
    \begin{equation*}
        \int_{a}^{t} e_{\alpha,\beta}(t-s,\lambda) (s-a)^{r}\, \mathrm{d}s = \Gamma(r+1)e_{\alpha,\beta+r+1}(t-a,\lambda).
    \end{equation*}
\end{lemma}
To be more specific, we directly compute some special integral with respect to the integral of \(e_{\alpha,\beta}\) function
\begin{equation*}
    \begin{aligned}
        \int_{a}^{b} e_{\alpha,\beta} (t-s,\lambda) \,\mathrm{d}s = &e_{\alpha,\beta +1} (t-a,\lambda) - e_{\alpha,\beta+1}(t-b,\lambda)\\
        \int_{a}^{b} e_{\alpha,\beta} (t-s,\lambda) (s-a) \,\mathrm{d}s = &e_{\alpha,\beta +2} (t-a,\lambda) + e_{\alpha,\beta+2}(t-b,\lambda) \\
        &-(b-a) e_{\alpha,\beta+1}(t-b,\lambda)
    \end{aligned}
\end{equation*}

We approximate \(F(U(r))\) by the constant \(F(U_{j})\), we get the EI method
\begin{equation}\label{ei_weights_scheme}
    U^{n} = e_{\alpha,1} (n \tau , K) U^{0} + \sum\limits_{j=0}^{n-1} W_{n-j}^{(1)} F^{j},
\end{equation}
where \(W_{j} = e_{\alpha, \alpha + 1} (j \tau ,K) - e_{\alpha,\alpha + 1} ( (j-1) \tau , K)\) represent the \textbf{one step} scheme weights and \(F^{j} = F(U^{j}) \) represent the approximation of the semilinear term \( F(U(t_{j})) \).

The one step EI numerical scheme can be enhanced by employing polynomials of higher degree \cite{garrappa2011accurate}. We could consider the piecewise first-order interpolating polynomials
\begin{equation*}
    p_{j}(r) = F^{j} + \frac{r-t_{j}}{\tau} ( F^{j+1} - F^{j}),
\end{equation*}
we replace \(F(U(r))\) by \(p_{j}(r)\) when \(r \in [t_{j},t_{j+1}] , j=0,1,\cdots,n-2\) and by \(p_{n-2}(r)\) when \(r\) lies in the last interval \([t_{n-1},t_{n}]\). Thus, we obtain the two step EI scheme by
\begin{equation}\label{ei_twostep_scheme}
    U^{n} = e_{\alpha,1}(n\tau,K)U^{0} + \widetilde{W}_{n}F^{0} + \frac{1}{\tau} \sum\limits_{j=1}^{n-1} W_{n-j}^{(2)} F^{j} + \frac{e_{\alpha,\alpha+2}(\tau,K)}{\tau}(2F^{n-1} - F^{n-2}) \ n \geq 2,
\end{equation}
where the weights of the \textbf{two step} scheme is that \(W_{j}^{(2)} = e_{\alpha,\alpha+2} ((j-1)\tau,K) - 2 e_{\alpha,\alpha + 2} (j\tau,K) + e_{\alpha,\alpha +2}((j+1)\tau,K)\) and \(\widetilde{W}_{n} = \frac{1}{\tau} \bigl( e_{\alpha,\alpha+2}((n-1)\tau , K) - e_{\alpha,\alpha+2}((n\tau,K)) \bigr) + e_{\alpha,\alpha+1}(n\tau,K)\).

It is not feasible to compute the matrix function \(e_{\alpha , \beta} (t,K)\) directly. We note that the coarse scale matrix \(M, A\) is diagonalizable, thus we can take the matrix \(K\) as \(K = Q^{-1} \Lambda Q\) where the columns of \(Q\) are the eigenvectors of \(K\) or the generalized eigenvectors of \(A\) with respect to \(M\) and \(\Lambda\) is a diagonal matrix with the eigenvalues of \(K\). We substitute the matrix function \(e _{\alpha , \beta} (t,K) = Q^{-1} e _{\alpha , \beta} (t,\Lambda) Q\) into the fully discretized scheme with exponential integrator (EI). Even though the upscaled homogenized system has already reduced the dimension significantly, we could still improve the efficiency of the eigen-decomposition. Some techniques such as FFT method, Conjugate Gradient (CG) method are used to accelerate the computations of eigenvalue problems. Finally, the discretized scheme with one step EI \eqref{ei_weights_scheme} can be described as
\begin{equation}\label{exponential_integrator_scheme}
    Q^{-1} {U}^{n} = e_{\alpha,1}(n\tau, \Lambda) Q {U}^{0} + e_{\alpha,\alpha +1} (n\tau, \Lambda) Q{F}^{0}  + \sum\limits_{j=1}^{n-1} e_{\alpha, \alpha +1} \bigl( (n-j)\tau ,\Lambda \bigr)  Q ({F}^{j} - {F}^{j-1}).
\end{equation}
Similar for the two step EI \eqref{ei_twostep_scheme}.

\subsection{Error estimate for the discretized scheme}

In this subsection, we discuss the convergence of the numerical scheme \eqref{exponential_integrator_scheme}. The error between the numerical solution and the exact solution mainly comes from the multicontinuum model and the exponential integrator scheme. We have already obtained the former in Section \ref{Space_discretization} and the latter, \(\ie\) temporal discretization error, results from the numerical integration. We first give the estimate of the weight \(W_{n,j}\) of \eqref{ei_weights_scheme} in the following Lemma by the property of Mittag-Leffler function.
\begin{lemma}\label{weight_estimation}
    There exists a constant \(C\) only depending on \(\alpha, \ K\) such that for any \(n \geq 1\) \cite{garrappa2013exponential}
    \begin{equation*}
        \bigl\lVert W_{n,j} \bigr\rVert \leq C \tau ^{\alpha} (n-j) ^{\alpha -1} \ \ j = 1,2,\cdots,n-1.
    \end{equation*}
\end{lemma}

Assume that the homogenized solution \( \bigl( U_{0,h,H}(x,t) , U_{1,h,H}(x,t) \bigr) \) is sufficiently smooth at time \(t_{n}\). Assume the solution \(U_{i,h,H}(t),\ i=0,1\) can be expanded in mix powers of integer and fractional order \cite{lubich1982runge} according to \(U_{i,h,H}(t) = C_{0,0} + C_{0,1}t^{\alpha} + C_{0,2}t^{2\alpha} + C_{1,0}t + C_{1,1}t^{1+\alpha} + C_{1,2}t^{1+2\alpha} + \cdots\). We could obtain the convergence of the exponential integrator scheme \eqref{ei_weights_scheme} from the following theorem.
\begin{theorem}
    Let \( ( U_{0}^{n}, U_{1}^{n} )\) be the approximation of homogenized solution \( ( U_{0,h,H}(t_{n}) ,\\ U_{1,h,H}(t_{n}) ) \). We have following error estimates
    \begin{equation}\label{full_discretized_error}
            \max\limits_{1\leq n \leq N} \bigl\lVert \frac{1}{H^{d}} \int_{[0,H]^{d}} P_{h,H} ( U_{0,h,H}(t_{n}) , U_{1,h,H}(t_{n}) ) -  P_{h,H} ( U_{0}^{n} , U_{1}^{n} ) \bigr\rVert \leq C \tau ^{1+\alpha}.
    \end{equation}
\end{theorem}

\begin{proof}
    We denote \( U(t) = \bigl( U_{0,h,H} (t) , U_{1,h,H} (t) \bigr) \) and \( U^{n} = (U_{0}^{n} , U_{1}^{n}) \). For the second term, by the definition of the NLMC downscaling operator \(P_{h,H}\), we derive that
    \begin{equation*}
        \begin{aligned}
            &\bigl\lVert \frac{1}{H^{d}} \int_{[0,H]^{d}} P_{h,H} \bigl( (U_{0,h,H}(t_{n}) , U_{1,h,H}(t_{n})) - ( U_{0}^{n} , U_{1}^{n}) \bigr) \bigr\rVert\\
            \leq &C \bigl\lVert (U_{0,h,H}(t_{n}) , U_{1,h,H}(t_{n})) - ( U_{0}^{n} , U_{1}^{n}) \bigr\rVert = C \bigl\lVert U(t_{n}) - U^{n} \bigr\rVert.
        \end{aligned}
    \end{equation*}
    By \eqref{exponential_integrator_solution}, we have
    \begin{equation*}
        U(t) = e_{\alpha,1} (t,K) U(0) + \sum\limits_{j=0}^{n-1} W_{n,j} F(U(t_{j})) + E_{n},
    \end{equation*}
    where \(E_{n}\) is the error included by numerical integration
    \begin{equation*}
        E_{n} = \sum\limits_{j=0}^{n-1} \int_{t_{j}}^{t_{j+1}} e_{\alpha,\alpha}(t_{n} - a;K) \bigl( F(U(r)) - F(U(t_{j})) \bigr) \mathrm{d}r,
    \end{equation*}
    and
    \begin{equation*}
        \begin{aligned}
            \bigl\lVert E_{n} \bigr\rVert &\leq L \mathrm{cond}(Q) \sum\limits_{j=0}^{n-1} \int_{t_{j}}^{t_{j+1}} \bigl\lVert e_{\alpha,\alpha}(t_{n} - r;\Lambda) \bigr\rVert \, \cdot \,\bigl\lVert U(r) - U(t_{j}) \bigr\rVert\,\mathrm{d}r\\
            &\leq C_{1} \sum\limits_{j=0}^{n-1} \int_{t_{j}}^{t_{j+1}} (t_{n} - r)^{\alpha -1} \bigl\lVert U(r) - U(t_{j}) \bigr\rVert\,\mathrm{d}r,
        \end{aligned}
    \end{equation*}
    where \(L\) represents the Lipschitz constant. By the mixed expansion in mixed powers of \(t\) and \(t^{1+\alpha}\).\cite{lubich1982runge,garrappa2011accurate} \(\bigl\lVert U(t_{j+1}) - U(t_{j}) \bigr\rVert\) can be bounded by the order \( \tau ^{\alpha}\) in \([0,T]\). Then we have
    \begin{equation*}
        \bigl\lVert E_{n} \bigr\rVert \leq C_{1} \tau ^{1 + \alpha} t_{n}^{\alpha - 1} + C_{2} \tau.
    \end{equation*}
    With the help of Lemma \refeq{weight_estimation} and the discrete Gr\"onwall's inequality \cite{dixon1985order}
    \begin{equation*}
        \begin{aligned}
            \bigl\lVert U(t_{n}) - U^{n} \bigr\rVert &\leq \bigl\lVert E_{n} \bigr\rVert + \sum\limits_{j=0}^{n-1} \bigl\lVert W_{n,j} \bigr\rVert  \bigl\lVert F(U(t_{j})) - F(U^{j}) \bigr\rVert\\
            &\leq \bigl\lVert E_{n} \bigr\rVert + L C' \tau ^{\alpha} \sum\limits_{j=0}^{n-1} (n-j)^{\alpha - 1} \bigl\lVert U(t_{j}) - U^{j} \bigr\rVert\\
            &\leq C \bigl\lVert E_{n} \bigr\rVert \leq C \bigl( \tau + t_{n} ^{\alpha -1} \tau ^{1+\alpha} \bigr).
        \end{aligned}
    \end{equation*}
    Combining all the results above, we give the final estimate \eqref{full_discretized_error}.

\end{proof}

\section{Numerical Experiments}\label{Numerical_experiments}

In this section, we will show some numerical tests for the time fractional parabolic equation with the multiscale permeability field \(\kappa(x)\). We take the domain \(\Omega = [0,1] \times [0,1] \). For the reference solution, we utilize the implicit L1 scheme with finer time step where the time interval \([0,T]\) is divided into uniform parts \(\widehat{N} = 5N\) for time discretization, and the finite element method approximated by quadratic elements basis in space with mesh size \(\frac{1}{400}\) for spatial discretization. We derive the reference solution numerical scheme by
\begin{equation}
    ( M_{f} + \alpha _{0} A_{f}) u^{n+1} = \alpha_{0} f(u^{n+1}) +  (1-b_{1})M_{f} u^{n} + M_{f} \sum\limits_{j=1}^{n-1} (b_{j} - b_{j+1}) u^{n-j} + b_{n} M_{f} u^{0},
\end{equation}
where \(M_{f},A_{f}\) represent the mass matrix and the stiffness matrix in the fine scale. We will discuss the linear problem with source term \(f(x)\) and semilinear source term \(f(x;u(x))\) with two different media respectively. 
In all examples, we take
\begin{equation*}
    \kappa  (x) = \frac{\varepsilon}{10^{5}} \quad \mathrm{in} \ \Omega_{0}, \quad \kappa  (x) = \frac{1}{100\varepsilon} \quad \mathrm{in} \ \Omega _{1}.
\end{equation*}

In all examples, we set the coarse mesh grid size as \(H = \frac{1}{20}\), \(T = 0.001 \) and \(N = 100\). The relative errors between the reference solutions and numerical solutions at each time \(t =  t_{n}\) are defined as follows.
\begin{equation}
    e_{i}^n = \frac{\bigl\lVert \Pi _{i,h} u^{n}  - U_{i}^{n} (x) \bigr\rVert }{\bigl\lVert \Pi _{i,h}  u^{n} \bigr\rVert}, \quad i=0,1.
\end{equation}
where \(\Pi _{i,h} u^{n}\) is the average of reference solutions \(u^{n}\) (solved by finite element method in the fine grid) while \(U_{i}^{n}\) denotes the approximated solution from our proposed method at \(t_{n}\) for a specific continua \(i\). A logarithmic scale has been used for error-axis in all relative error figures. To compare, for the numerical solution, we also present the result of explicit L1 scheme that are unstable for some cases in order to highlight the stability of our methods we present. The explicit L1 scheme reads as
\begin{equation}\label{Explicit_L1_fully_discretized}
    M U^{n+1} = \alpha_{0} f(U^{n+1}) + \bigl( \alpha _{0}A + (1-b_{1})M \bigr) U^{n} + M \sum\limits_{j=1}^{n-1} (b_{j} - b_{j+1}) U^{n-j} + b_{n} M U^{0}.
\end{equation}

\subsection{Smooth source term}
We first take a smooth source term \(f( x_{1},x_{2} ) = \exp \bigl( -50 ( (x_{1} - 0.5)^{2} + (x_{2} - 0.5)^{2}  ) \bigr)\) and the initial condition \(u_{0}(x_{1},x_{2}) = 5 \times 10^{-3} \sin ( 2 \pi x_{1} ) \sin (\pi x _{2})\). In this example, we set \(\alpha = 0.9, 0.6, 0.3\) and \(\varepsilon = \frac{1}{10}\). The permeability field \(\kappa \) is depicted in the left of Figure \ref{kappa_f}, which is a periodic field. The blue region represents the low conductivity region \(\Omega _{0}\), while the yellow region represents the high conductivity region \(\Omega _{1}\). We plot the reference solution snapshots at \(t=0, \ t=\frac{T}{2}, \ t=T\) respectively in Figure \ref{solution_snapshot_0.9}. In Figure \ref{U_0.9}, we show the upscaled solutions and the corresponding averaged reference solutions for \(\alpha=0.9\) at the final time. We note that our proposed approach provides an accurate approximation of the averaged reference solution for different \(\alpha\). Figure \ref{error_example1_linear} shows the error history in the temporal direction corresponding to different \(\alpha\). Using same time step size, we compare the approximate solution obtained using explicit L1 and our proposed method. The relative errors at \(t = T\) are shown in Table \ref{table_error}. It can be seen that for both methods, the approximation performs better when \(\alpha\) is closer to \(1\). However, we observe that the explicit L1 scheme does not converge to the reference solutions when \(\alpha\) goes to \(0\), while our method still gives good results, which demonstrates the effectiveness and stability of our proposed method.

\begin{figure}[H]
    \centering
    \subfloat
    {
        \begin{minipage}[t]{4.5cm}
            \centering
            \includegraphics[width=4cm, height = 3.5cm]{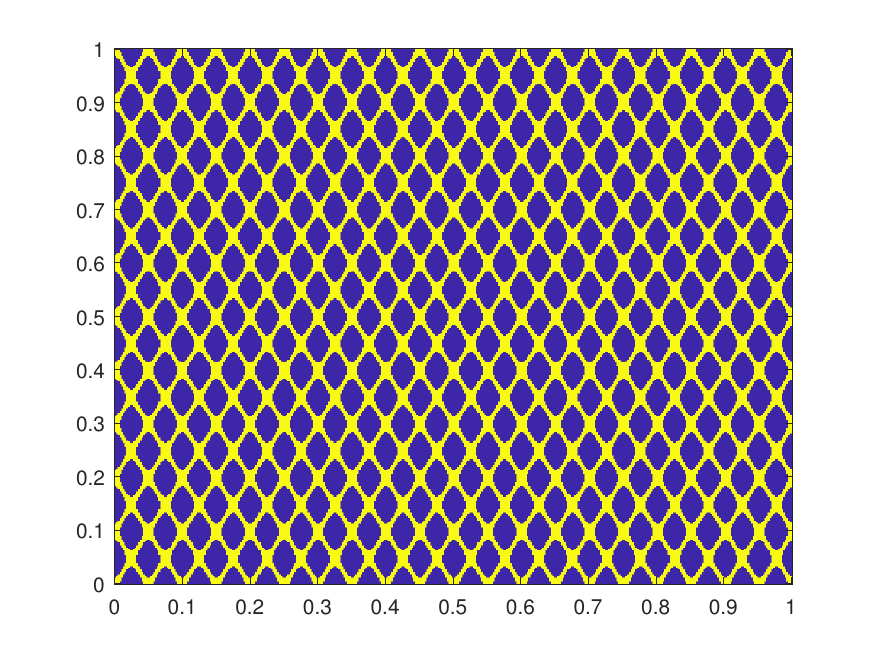}
        \end{minipage}
    }
    \subfloat
    {
        \begin{minipage}[t]{4.5cm}
            \centering
            \includegraphics[width= 4cm, height = 3.5cm]{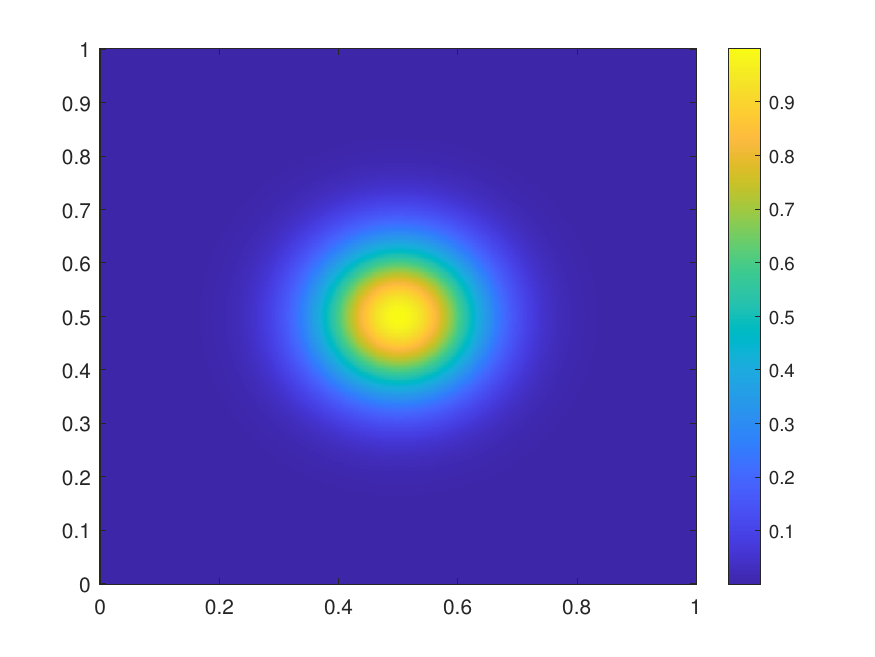}
        \end{minipage}
    }
    \caption{Left: \(\kappa \); Right: \(f\).}
    \label{kappa_f}
\end{figure}

\begin{figure}[H]
    \centering

    \includegraphics[width = 13.5cm,height = 3.5cm]{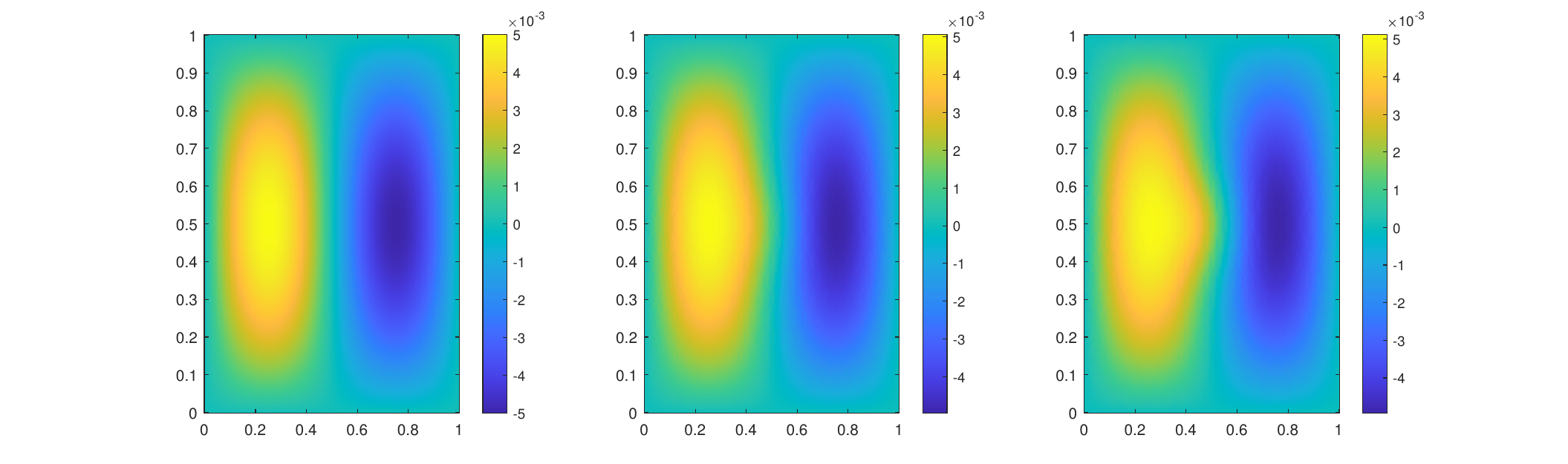}
    \caption{The solution snapshots to linear problems with \(\alpha = 0.9\), Left: \(t = 0\); Middle: \(t = \frac{T}{2}\); Right: \(t = T\).}
    \label{solution_snapshot_0.9}
    
\end{figure}

\begin{figure}[H]
    \centering

    \subfloat
    {
        \begin{minipage}[t]{4cm}
            \centering
            \includegraphics[width=4cm,height=3.5cm]{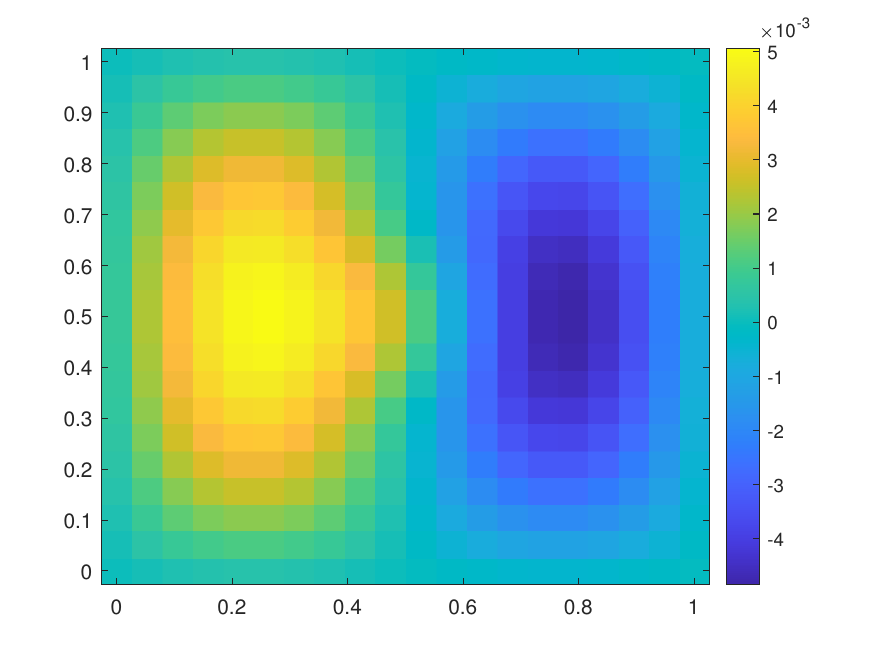}
        \end{minipage}
    }
    \subfloat
    {
        \begin{minipage}[t]{4cm}
            \centering
            \includegraphics[width=4cm,height=3.5cm]{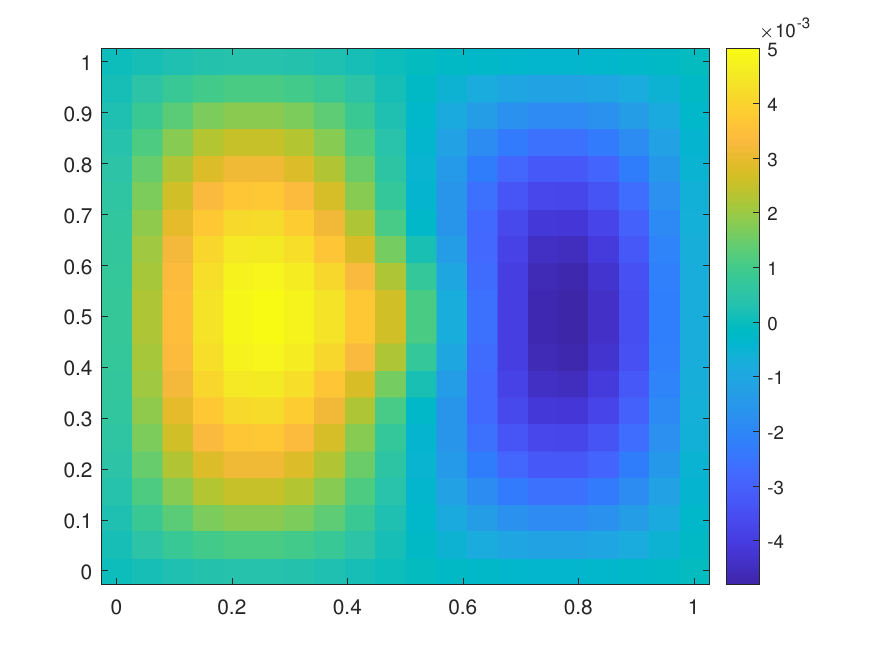}
        \end{minipage}
    }
    \subfloat
    {
        \begin{minipage}[t]{4cm}
            \centering
            \includegraphics[width=4cm,height=3.5cm]{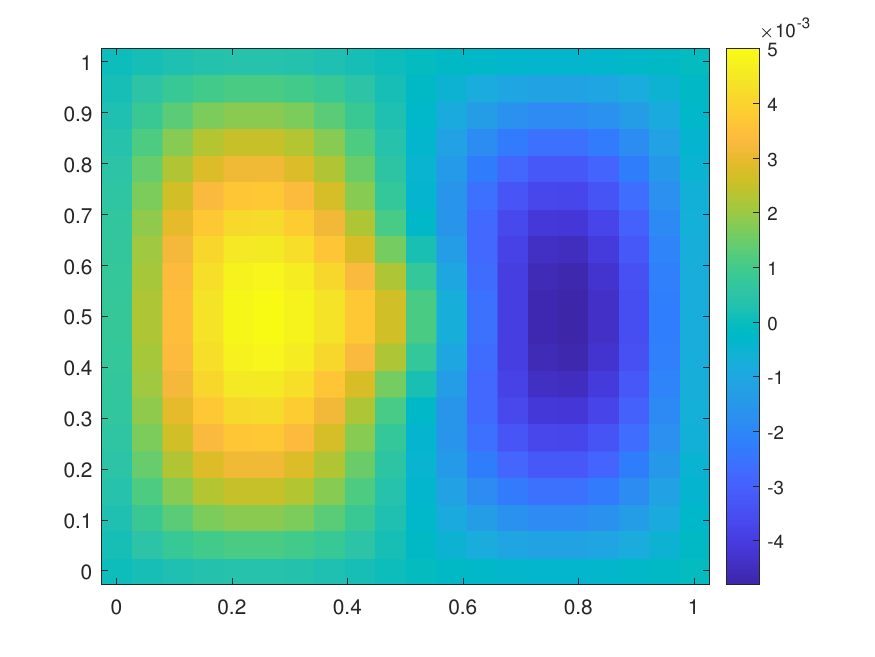}
        \end{minipage}
    }
    \\
    \subfloat
    {
        \begin{minipage}[t]{4cm}
            \centering
            \includegraphics[width=4cm,height=3.5cm]{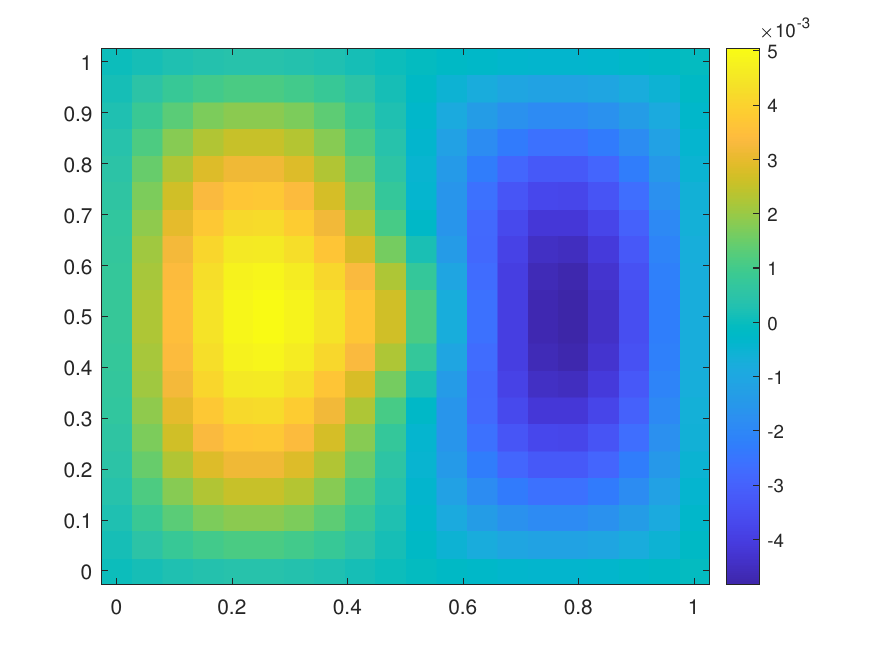}
        \end{minipage}
    }
    \subfloat
    {
        \begin{minipage}[t]{4cm}
            \centering
            \includegraphics[width=4cm,height=3.5cm]{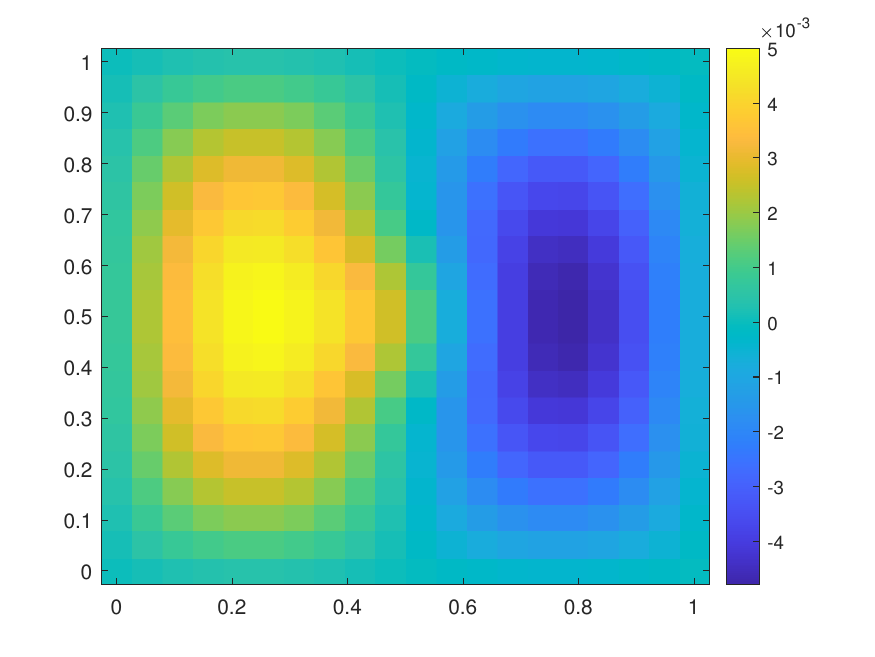}
        \end{minipage}
    }
    \subfloat
    {
        \begin{minipage}[t]{4cm}
            \centering
            \includegraphics[width=4cm,height=3.5cm]{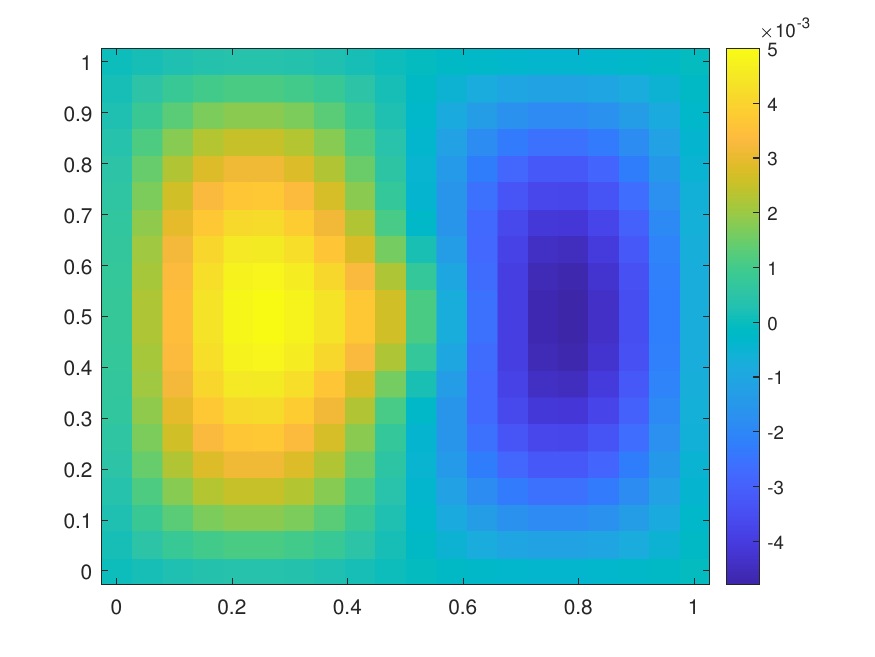}
        \end{minipage}
    }

    \caption{\(\alpha = 0.9 \) and \( t=T \). Left: average of reference solutions; Middle: numerical solutions by L1 implicit scheme; Right: numerical solutions by EI method. The first row: \(U_{0}\), the second row: \(U_{1}\).}
    \label{U_0.9}
\end{figure}

\begin{table}[H]
    \belowrulesep=0pt
    \aboverulesep=0pt
    \centering
    \begin{tabular}{c|cc|cc}
        \toprule
        \multirow{2}*{\(\alpha\)} & \multicolumn{2}{c|}{\(U_{0}\)} & \multicolumn{2}{c}{\(U_{1}\)}\\
         & Explicit L1 & EI method & Explicit L1 & EI method\\
        \midrule
        0.9 & 0.0105 & 0.0105 & 0.0083 & 0.0083\\
        0.6 & 0.0133 & 0.0122 & 0.0115 & 0.0130\\
        0.3 & NaN & 0.0462 & NaN & 0.0460\\
        \bottomrule
    \end{tabular}
    \caption{Relative errors at \(t=T\).}
    \label{table_error}
\end{table}

\begin{figure}[H]
    \centering
    
    \subfloat
    {
        \begin{minipage}[t]{5cm}
            \centering
            \includegraphics[width=4.5cm,height=4cm]{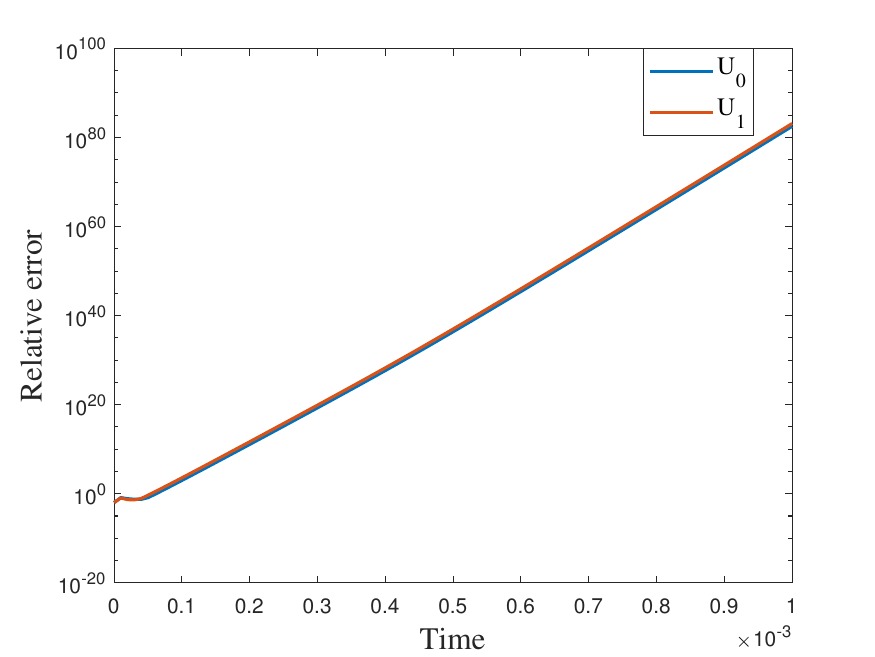}
        \end{minipage}
    }
    \subfloat
    {
        \begin{minipage}[t]{5cm}
            \centering
            \includegraphics[width=4.5cm,height=4cm]{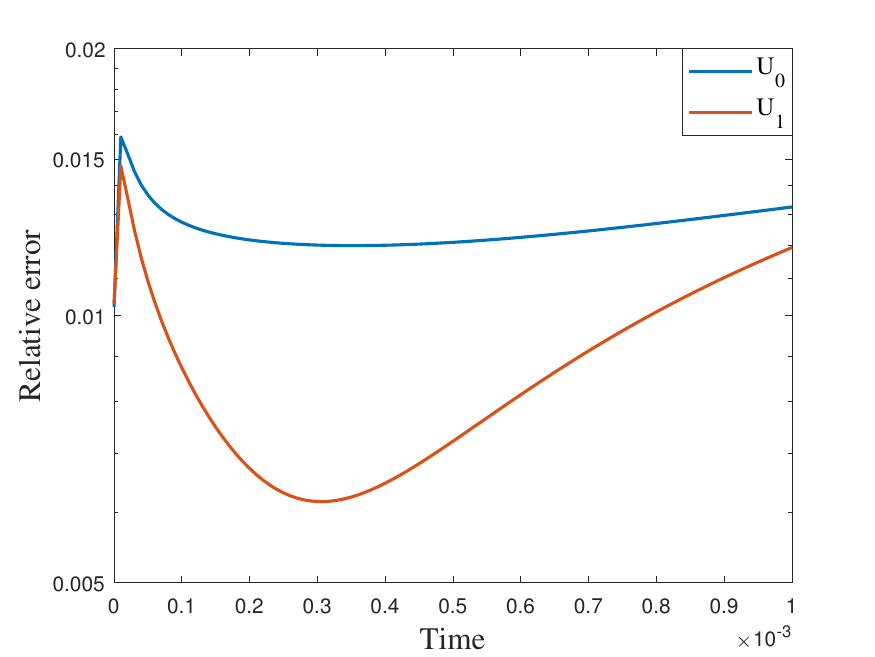}
        \end{minipage}
    }
    \subfloat
    {
        \begin{minipage}[t]{5cm}
            \centering
            \includegraphics[width=4.5cm,height=4cm]{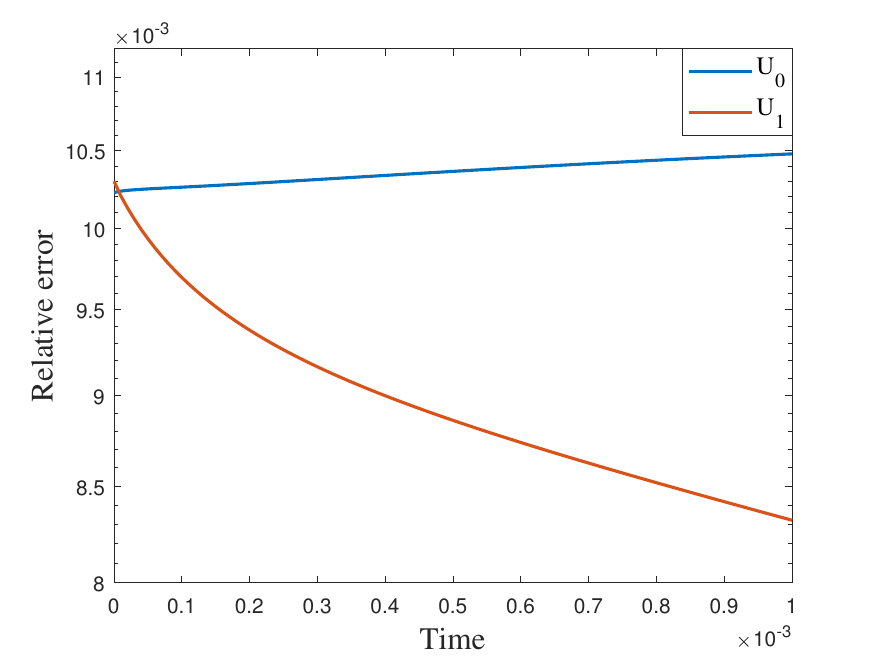}
        \end{minipage}
    }
    \\
    \subfloat
    {
        \begin{minipage}[t]{5cm}
            \centering
            \includegraphics[width=4.5cm,height=4cm]{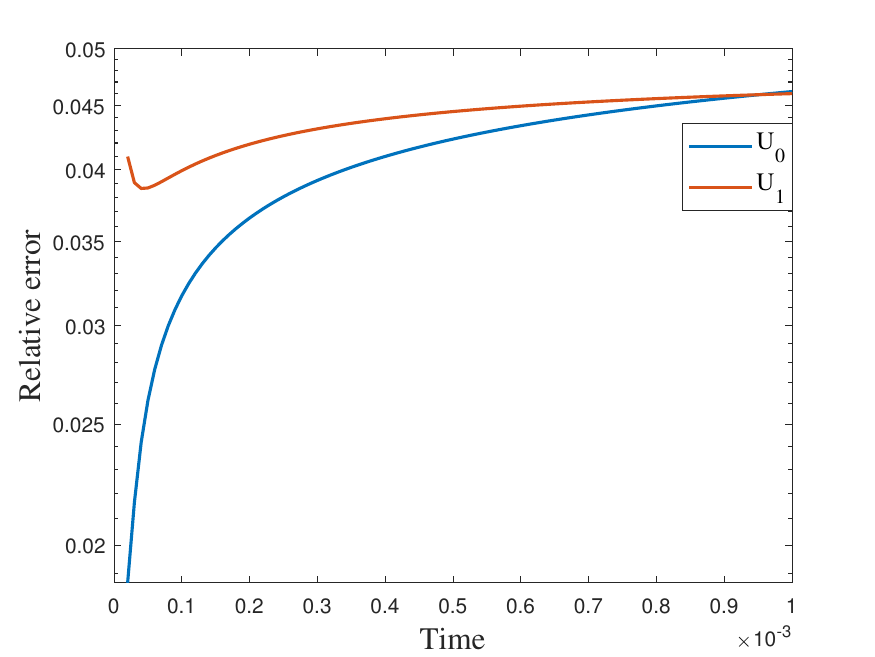}
        \end{minipage}
    }
    \subfloat
    {
        \begin{minipage}[t]{5cm}
            \centering
            \includegraphics[width=4.5cm,height=4cm]{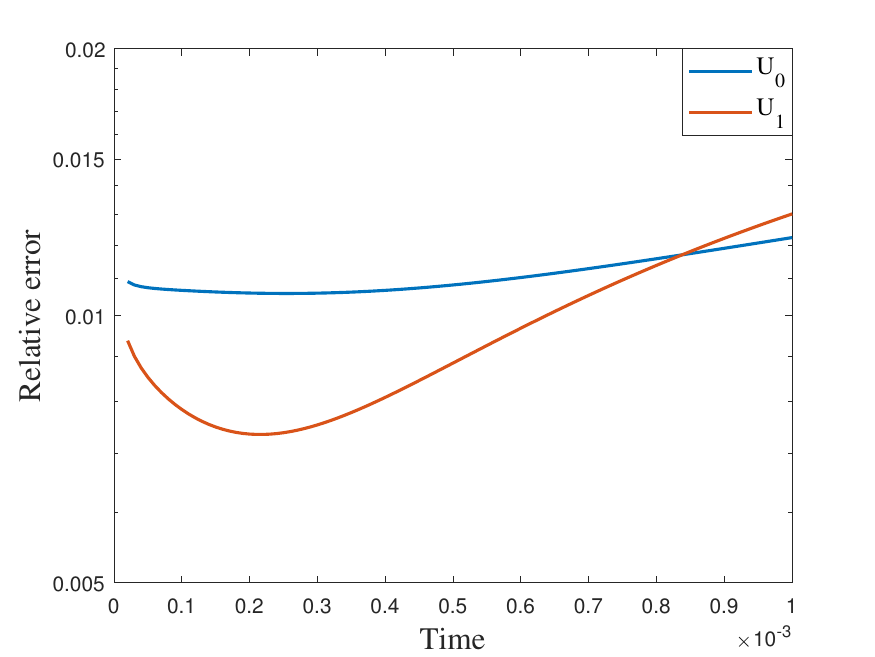}
        \end{minipage}
    }
    \subfloat
    {
        \begin{minipage}[t]{5cm}
            \centering
            \includegraphics[width=4.5cm,height=4cm]{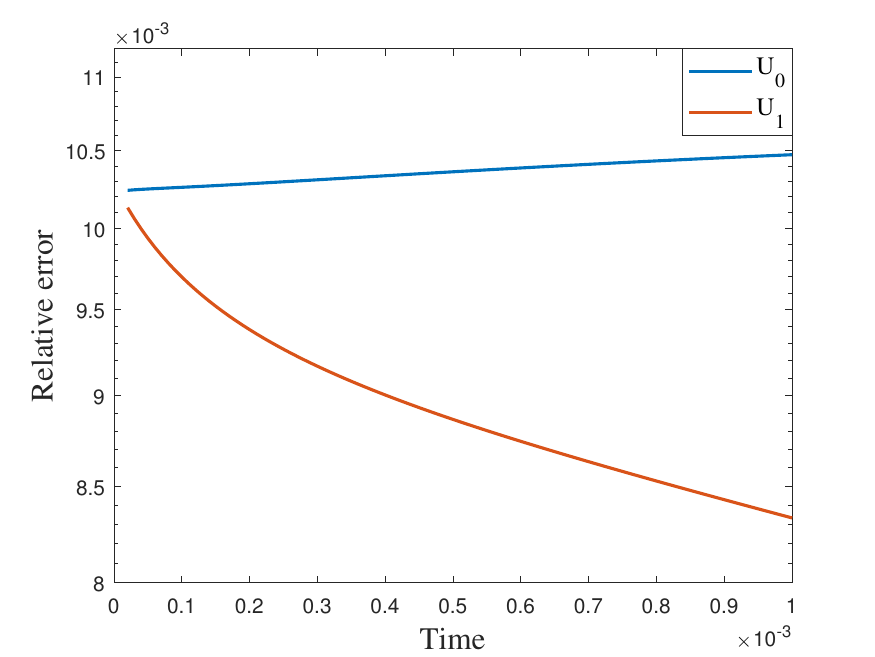}
        \end{minipage}
    }
    \caption{The first row: relative errors of explicit L1 scheme; The second row: relative errors of EI method. For each column, we take \(\alpha = 0.3, 0.6, 0.9\) form left to right.}
    \label{error_example1_linear}
\end{figure}

\subsection{Semilinear problems}

In this numerical test, we consider the semilinear problems and use the same media as in the first experiment. The results of the semilinear problems are also computed with different \(\alpha\) by our proposed method. We take \(f = u - u^{3} + f_{0}(x_{1},x_{2})\) with the initial condition \(u_{0}(x_{1},x_{2}) = -x_{1} (1-x_{1}) x_{2} (1-x_{2})\), where \(f_{0}(x_{1},x_{2}) = \exp \bigl(\,-50 \bigl(\, (x_{1} - 0.4)^{2} + (x_{2} - 0.6)^{2} \,\bigr)\,\bigr)\) is the smooth source term shown in the right of Figure \ref{kappa2_f2}. We present the relative errors for different \(\alpha = 0.9, 0.6, 0.3\). When \(\alpha\) tends to \(0\), the explicit L1 scheme turns to be unstable while the EI method performs well. Moreover, for the semilinear problems, the L1 scheme needs iterative methods at each time step while the exponential integrator method does not. The reference solution snapshots are depicted in Figure \ref{solution_snapshot_0.6_semi_k1}. The Figure \ref{U_0.6_semi_k1} illustrates the numerical solutions by L1 scheme and exponential integrator method compared to the average of the reference solutions. We figure out that our proposed method also works well for semilinear diffusion problems. We represent the relative errors in Figure \ref{error_example1_k1_semi} and Table \ref{table_error_semi} with different \( \alpha \).

\begin{figure}[H]
    \centering

    \includegraphics[width = 13.5cm,height = 3.5cm]{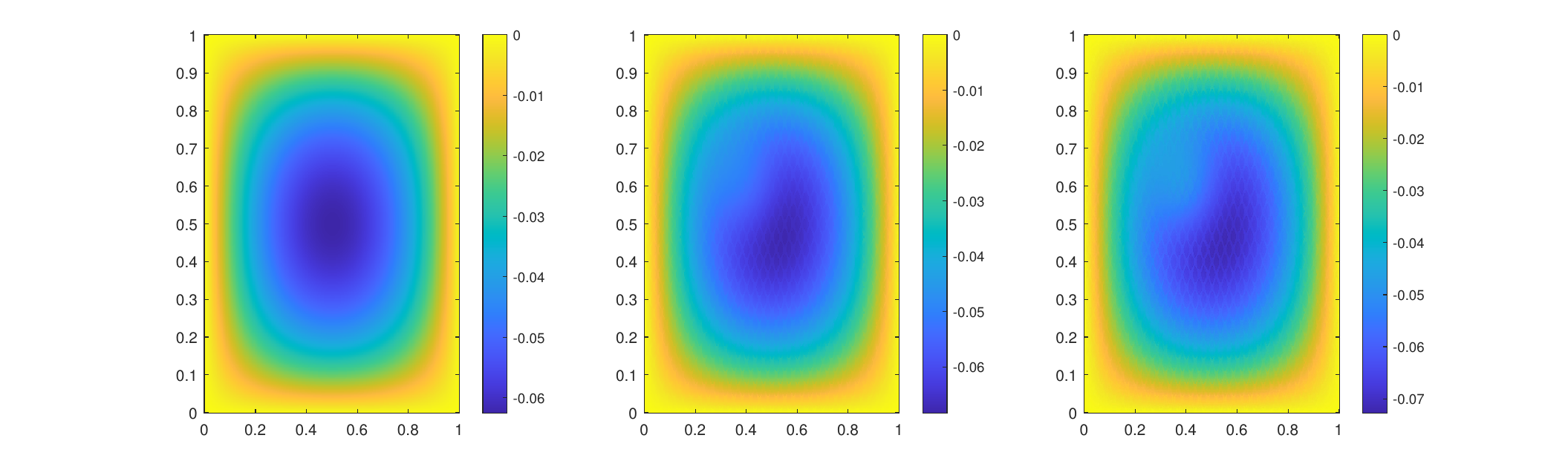}
    \caption{The solution snapshots to semilinear problem with \(\alpha = 0.6\), Left: \(t = 0\); Middle: \(t = \frac{T}{2}\); Right: \(t = T\).}
    \label{solution_snapshot_0.6_semi_k1}
    
\end{figure}

\begin{figure}[H]
    \centering

    \subfloat
    {
        \begin{minipage}[t]{4cm}
            \centering
            \includegraphics[width=4cm,height=3.5cm]{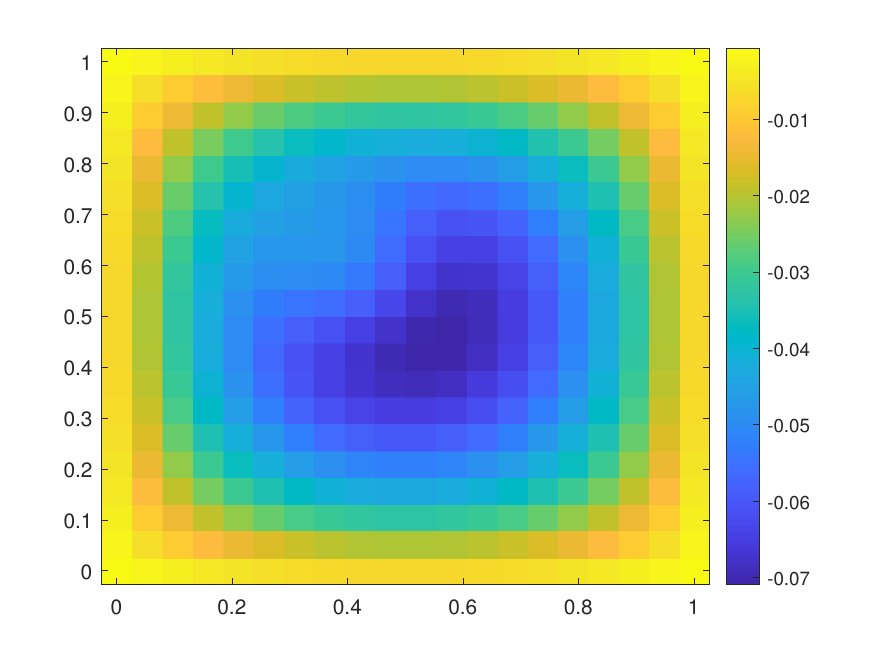}
        \end{minipage}
    }
    \subfloat
    {
        \begin{minipage}[t]{4cm}
            \centering
            \includegraphics[width=4cm,height=3.5cm]{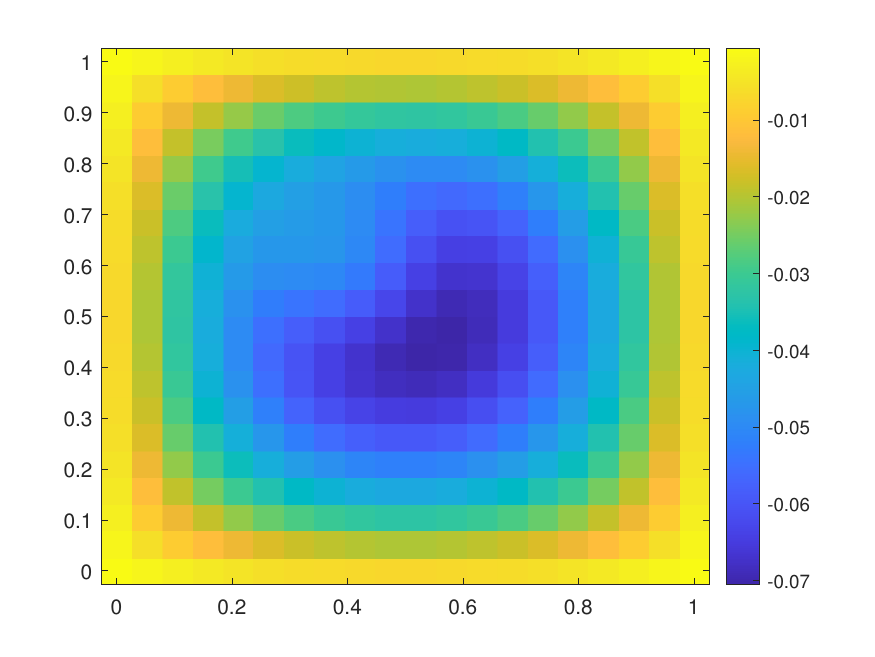}
        \end{minipage}
    }
    \subfloat
    {
        \begin{minipage}[t]{4cm}
            \centering
            \includegraphics[width=4cm,height=3.5cm]{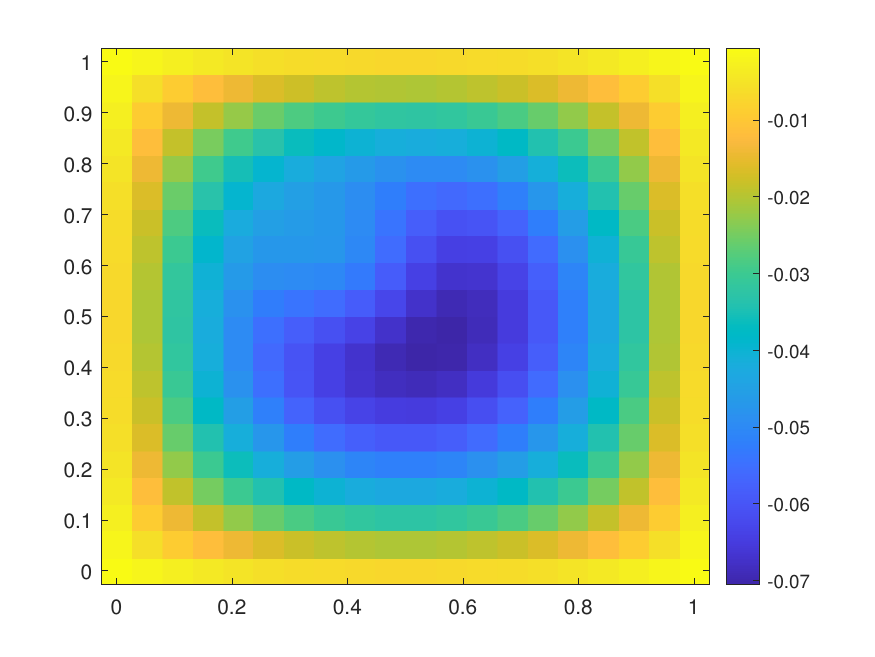}
        \end{minipage}
    }
    \\
    \subfloat
    {
        \begin{minipage}[t]{4cm}
            \centering
            \includegraphics[width=4cm,height=3.5cm]{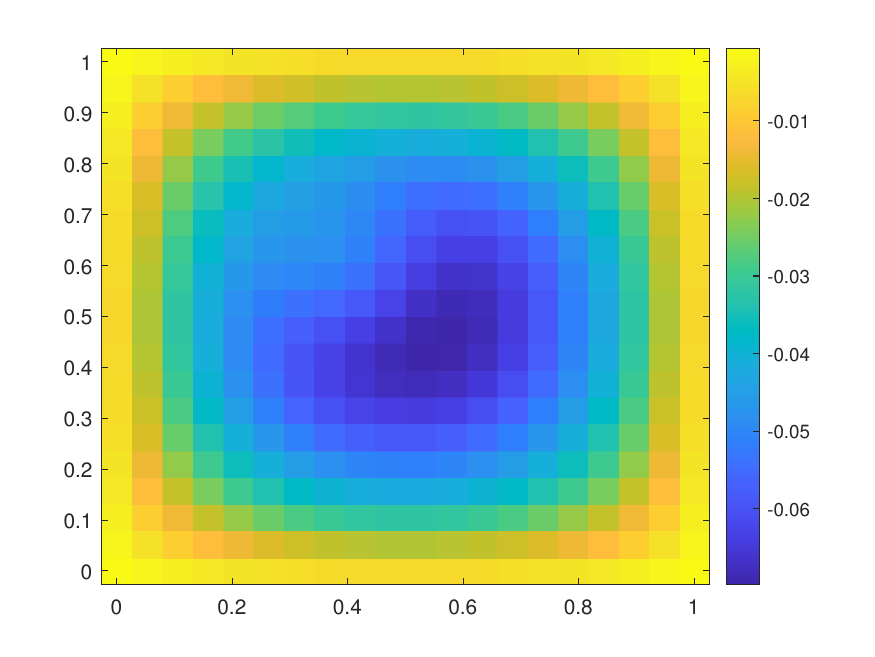}
        \end{minipage}
    }
    \subfloat
    {
        \begin{minipage}[t]{4cm}
            \centering
            \includegraphics[width=4cm,height=3.5cm]{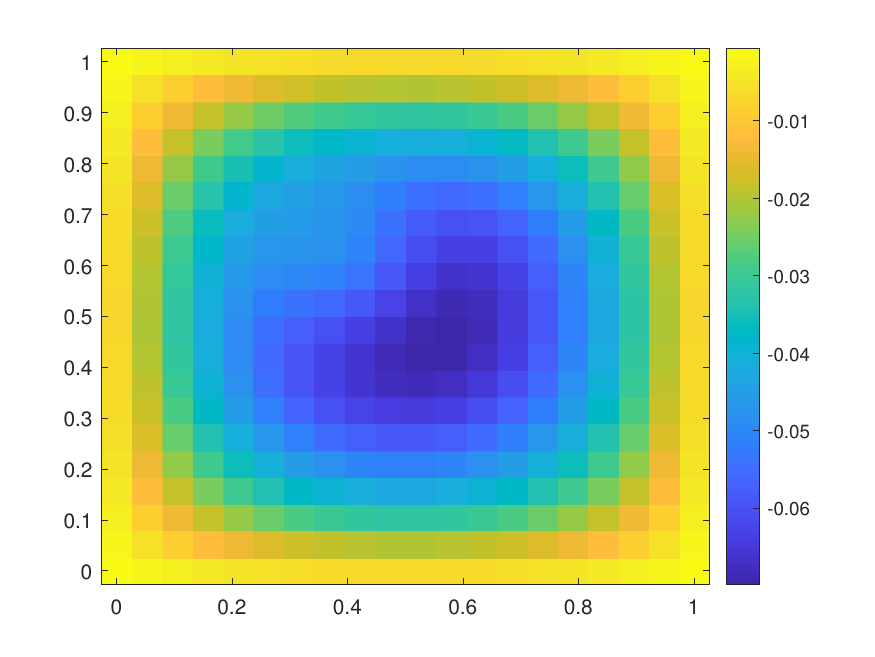}
        \end{minipage}
    }
    \subfloat
    {
        \begin{minipage}[t]{4cm}
            \centering
            \includegraphics[width=4cm,height=3.5cm]{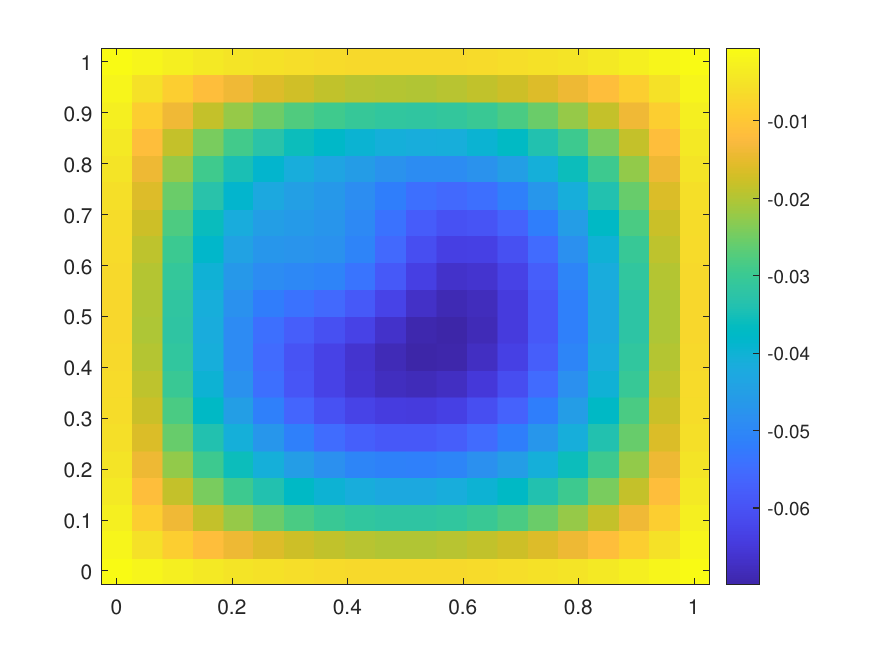}
        \end{minipage}
    }

    \caption{\(\alpha = 0.6 \) and \( t=T \). Left: average of reference solutions; Middle: numerical solutions by L1 implicit scheme; Right: numerical solutions by EI method. The first row: \(U_{0}\), the second row: \(U_{1}\).}
    \label{U_0.6_semi_k1}
\end{figure}

\begin{table}[H]
    \belowrulesep=0pt
    \aboverulesep=0pt
    \centering
    \begin{tabular}{c|cc|cc}
        \toprule
        \multirow{2}*{\(\alpha\)} & \multicolumn{2}{c|}{\(U_{0}\)} & \multicolumn{2}{c}{\(U_{1}\)}\\
         & Explicit L1 & EI method & Explicit L1 & EI method\\
        \midrule
        0.9 & 0.0165 & 0.0046 & 0.0035 & 0.0176\\
        0.6 & 0.2125 & 0.0070 & 0.2203 & 0.0019\\
        0.3 & NaN & 0.0607 & NaN & 0.0106\\
        \bottomrule
    \end{tabular}
    \caption{Relative errors at \(t=T\).}
    \label{table_error_semi}
\end{table}

\begin{figure}[H]
    \centering
    
    \subfloat
    {
        \begin{minipage}[t]{5cm}
            \centering
            \includegraphics[width=4.5cm,height=4cm]{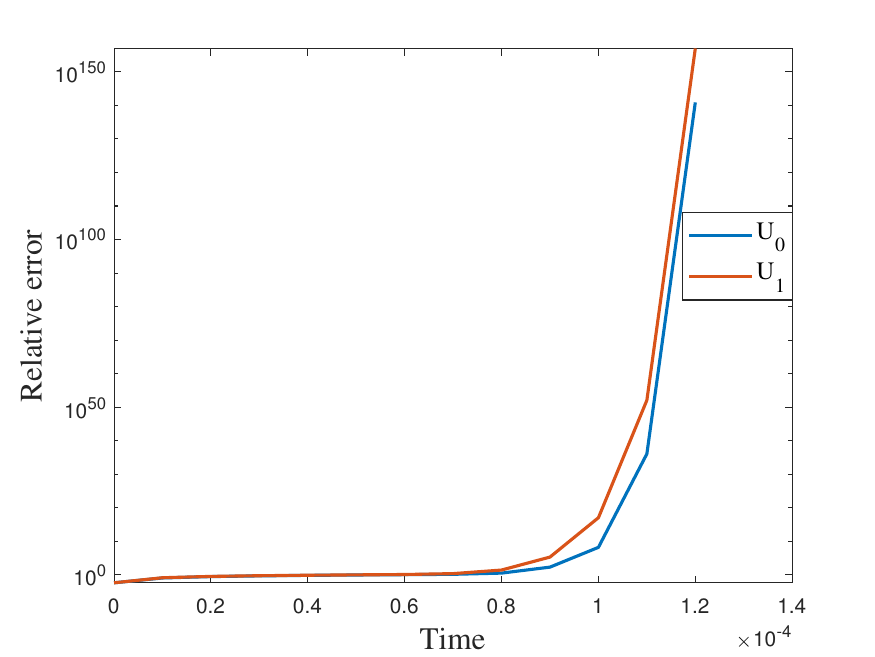}
        \end{minipage}
    }
    \subfloat
    {
        \begin{minipage}[t]{5cm}
            \centering
            \includegraphics[width=4.5cm,height=4cm]{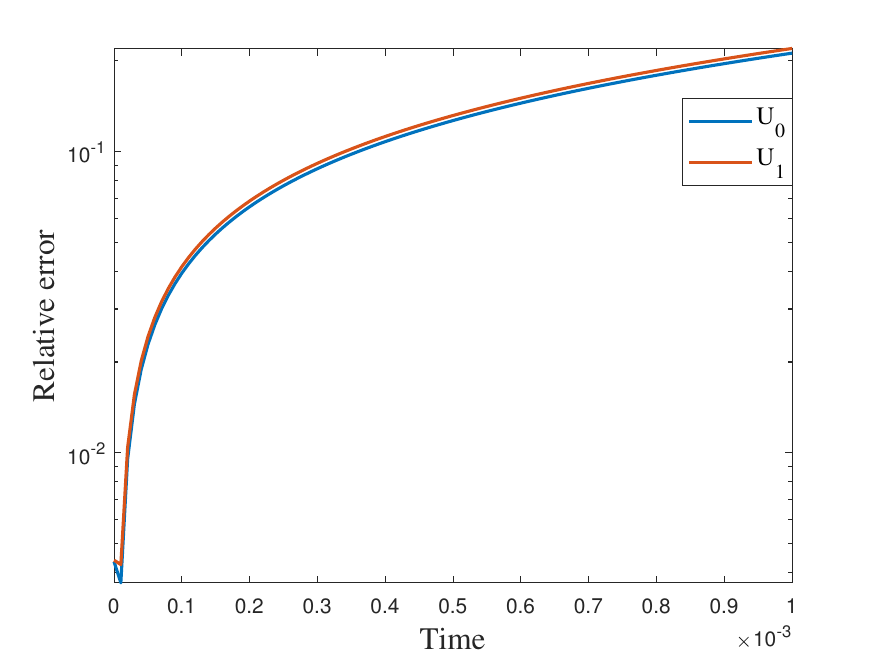}
        \end{minipage}
    }
    \subfloat
    {
        \begin{minipage}[t]{5cm}
            \centering
            \includegraphics[width=4.5cm,height=4cm]{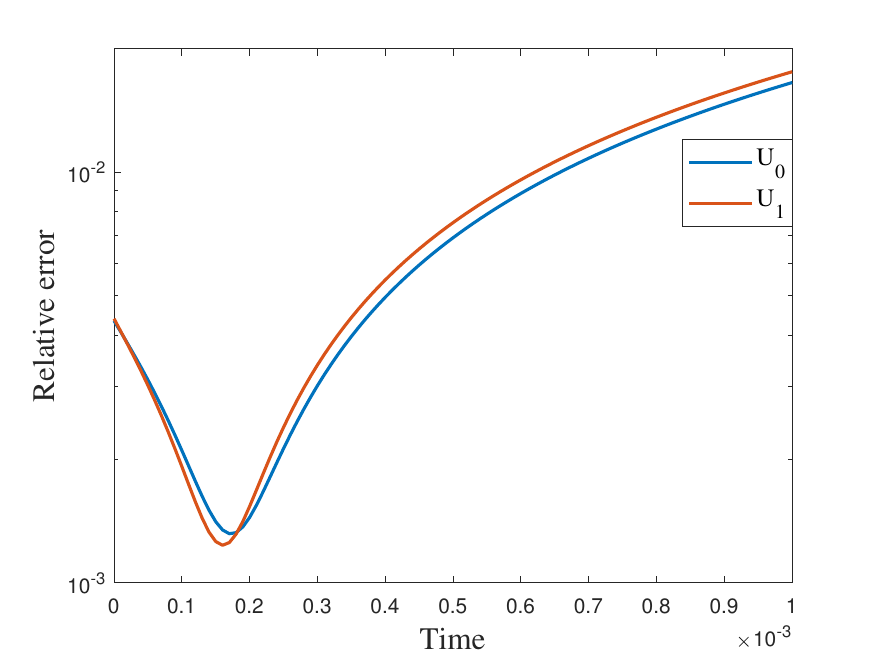}
        \end{minipage}
    }
    \\
    \subfloat
    {
        \begin{minipage}[t]{5cm}
            \centering
            \includegraphics[width=4.5cm,height=4cm]{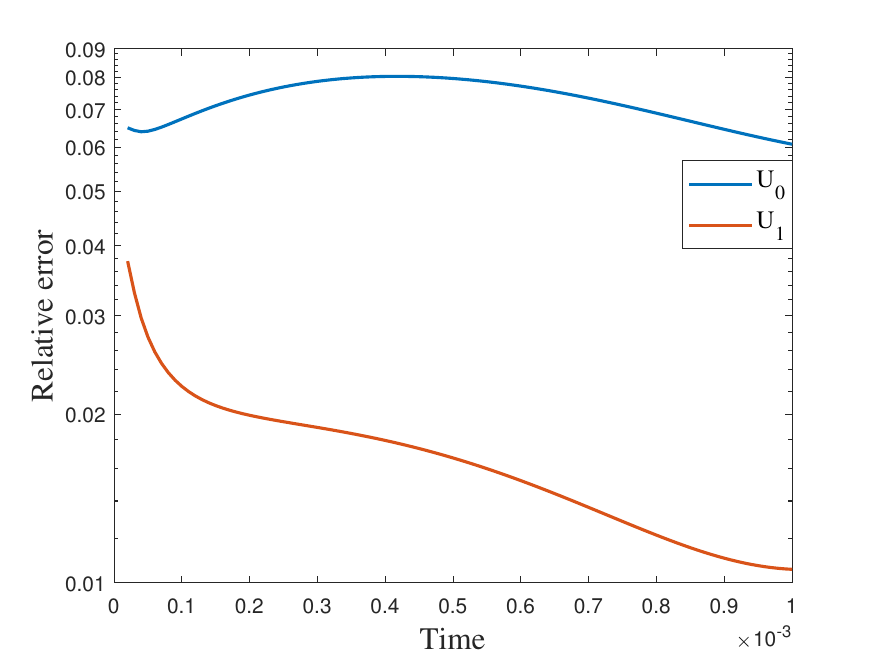}
        \end{minipage}
    }
    \subfloat
    {
        \begin{minipage}[t]{5cm}
            \centering
            \includegraphics[width=4.5cm,height=4cm]{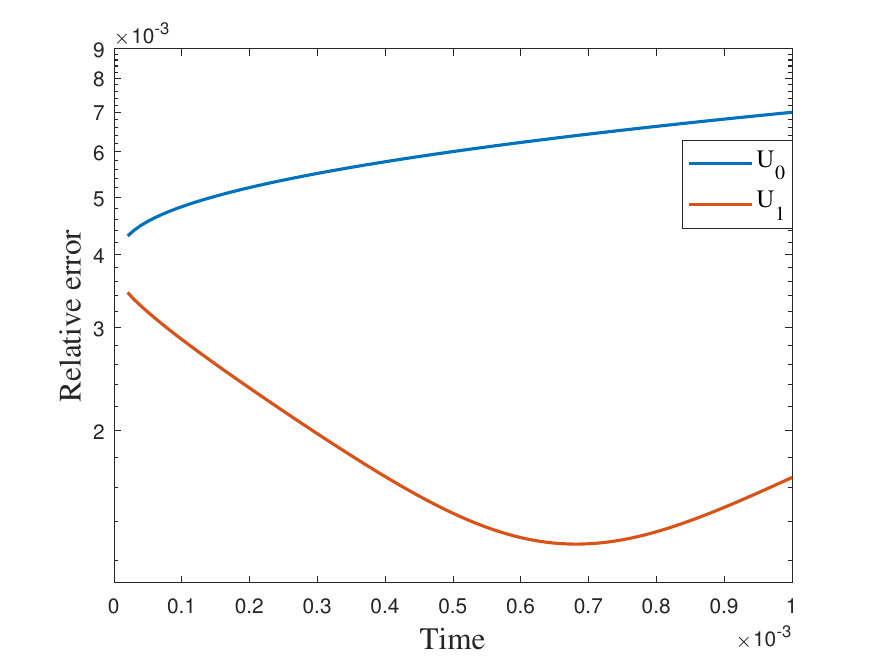}
        \end{minipage}
    }
    \subfloat
    {
        \begin{minipage}[t]{5cm}
            \centering
            \includegraphics[width=4.5cm,height=4cm]{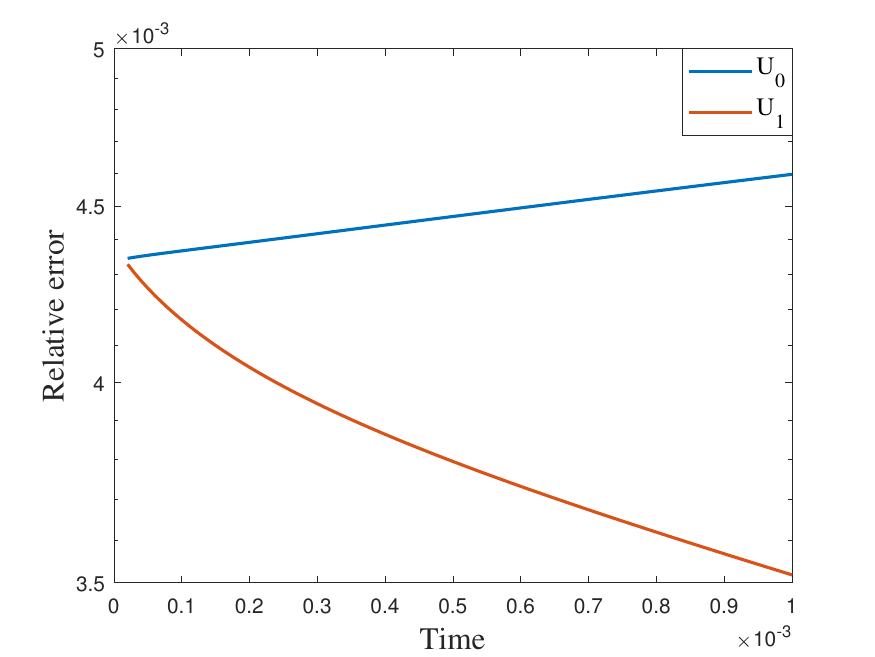}
        \end{minipage}
    }
    \caption{The first row: relative errors by explicit L1 scheme; The second row: relative errors by EI method. For each column, we take \(\alpha = 0.3, 0.6, 0.9\) form left to right.}
    \label{error_example1_k1_semi}
\end{figure}

\subsection{A more complicated media}
In our other example, we take nonperiodic media \(\kappa\) as shown in the left of Figure \ref{kappa2_f2}. The other settings are the same as the previous semilinear case. We also represent the reference solution snapshots in Figure \ref{solution_snapshot_0.6_semi_k2} with \(\alpha = 0.6\). By the relative errors in Figure \ref{error_example3_k2_semi} and Table \ref{table_error_example3}, the similar findings has been observed as in the previous cases. The errors of the numerical solutions is slightly larger compared to the previous examples, but they still converge to the reference solution.

\begin{figure}[H]
    \centering
    \subfloat
    {
        \begin{minipage}[t]{4.5cm}
            \centering
            \includegraphics[width=4.5cm, height = 3.5cm]{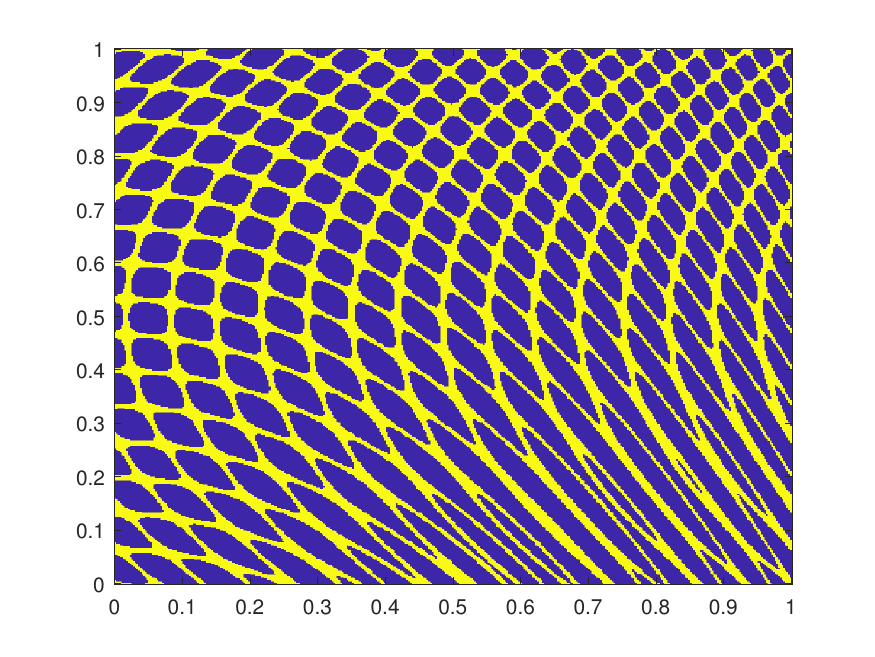}
        \end{minipage}
    }
    \subfloat
    {
        \begin{minipage}[t]{5cm}
            \centering
            \includegraphics[width=4.5cm, height = 3.5cm]{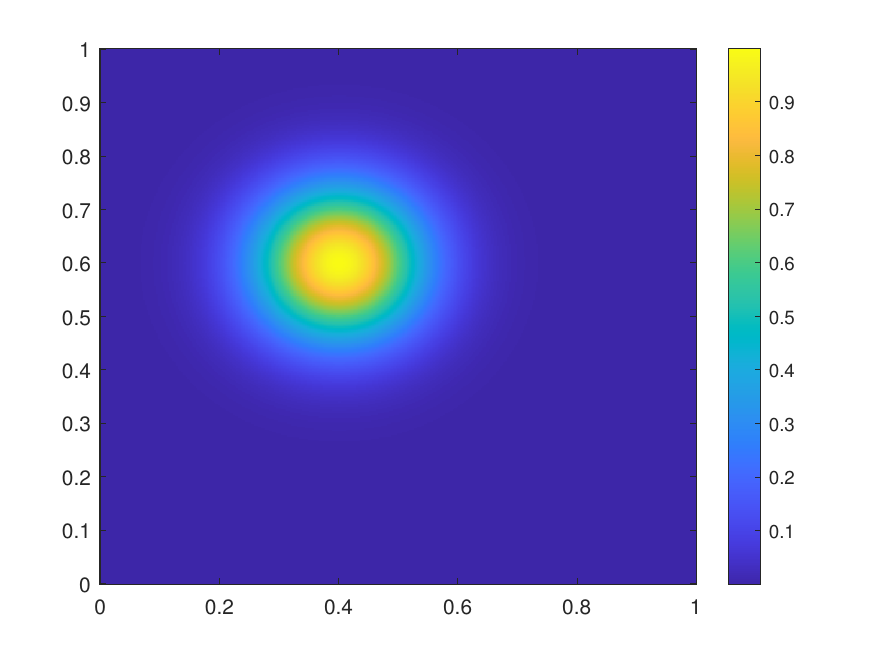}
        \end{minipage}
    }
    \caption{Left: \(\kappa\); Right: \(f_{0}\).}
    \label{kappa2_f2}
\end{figure}

\begin{figure}[H]
    \centering

    \includegraphics[width = 13.5cm]{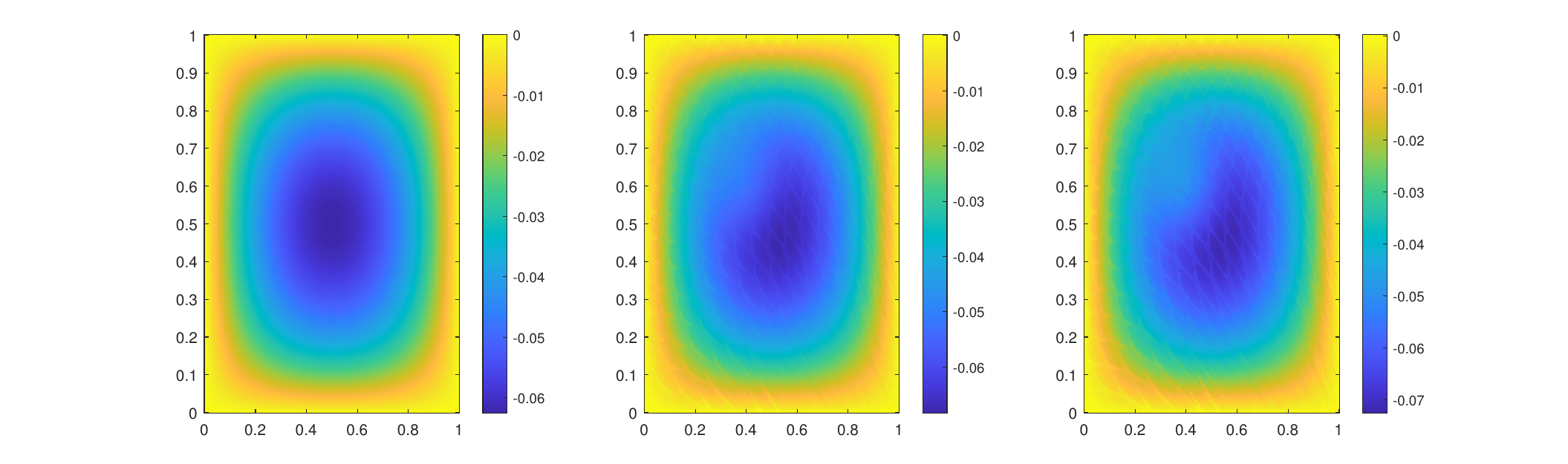}
    \caption{The solution snapshots to semi-linear problem with \(\alpha = 0.6\), Left: \(t = 0\); Middle: \(t = \frac{T}{2}\); Right: \(t = T\).}
    \label{solution_snapshot_0.6_semi_k2}
    
\end{figure}

\begin{table}[H]
    \belowrulesep=0pt
    \aboverulesep=0pt
    \centering
    \begin{tabular}{c|cc|cc}
        \toprule
        \multirow{2}*{\(\alpha\)} & \multicolumn{2}{c|}{\(U_{0}\)} & \multicolumn{2}{c}{\(U_{1}\)}\\
         & Explicit L1 & EI method & Explicit L1 & EI method\\
        \midrule
        0.9 & 0.0201 & 0.0122 & 0.0345 & 0.0308\\
        0.6 & 0.2133 & 0.0145 & 0.2137 & 0.0374\\
        0.3 & NaN & 0.1318 & NaN & 0.1173\\
        \bottomrule
    \end{tabular}
    \caption{Relative errors at \(t=T\).}
    \label{table_error_example3}
\end{table}

\begin{figure}[H]
    \centering
    
    \subfloat
    {
        \begin{minipage}[t]{5cm}
            \centering
            \includegraphics[width=4.5cm,height=4cm]{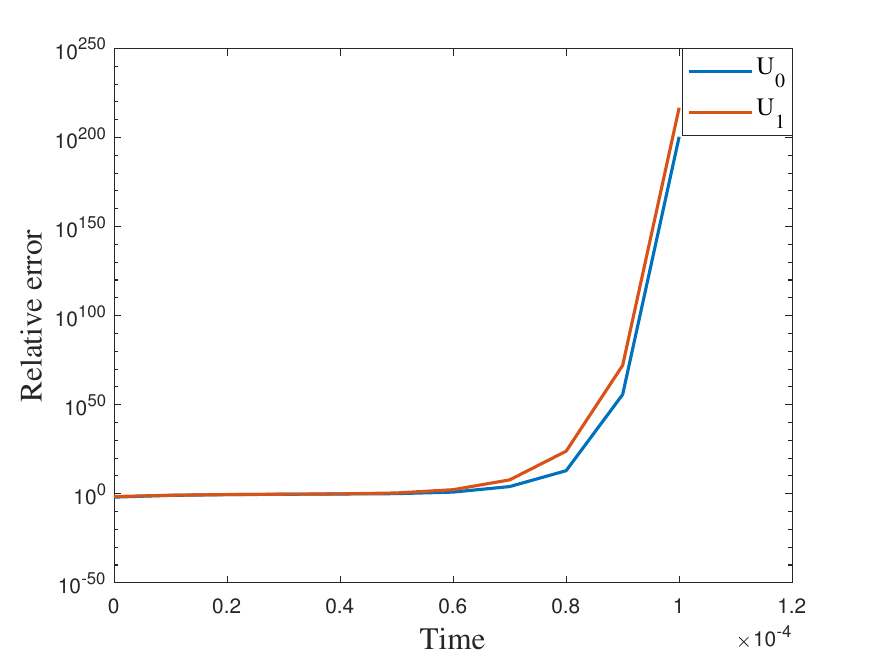}
        \end{minipage}
    }
    \subfloat
    {
        \begin{minipage}[t]{5cm}
            \centering
            \includegraphics[width=4.5cm,height=4cm]{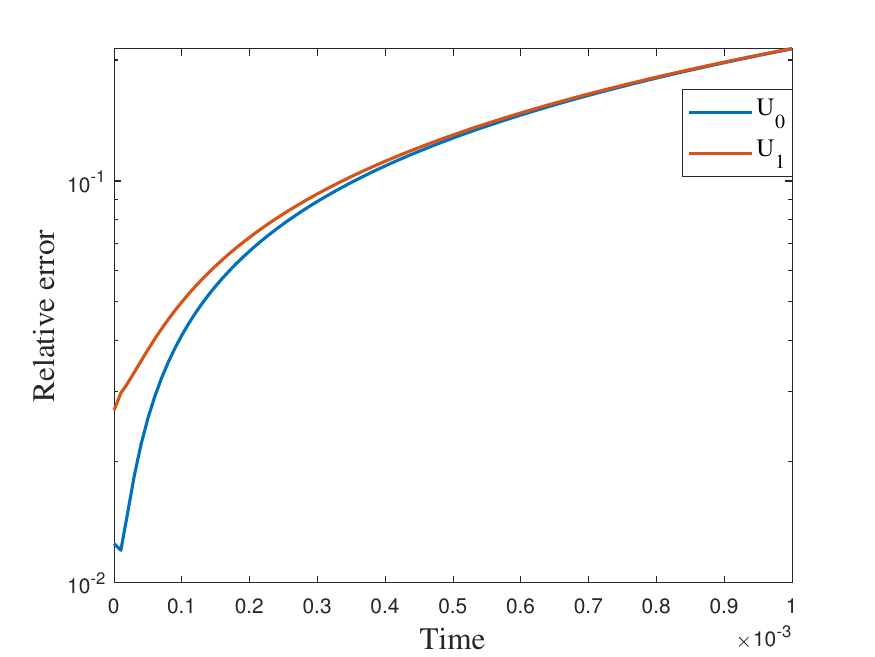}
        \end{minipage}
    }
    \subfloat
    {
        \begin{minipage}[t]{5cm}
            \centering
            \includegraphics[width=4.5cm,height=4cm]{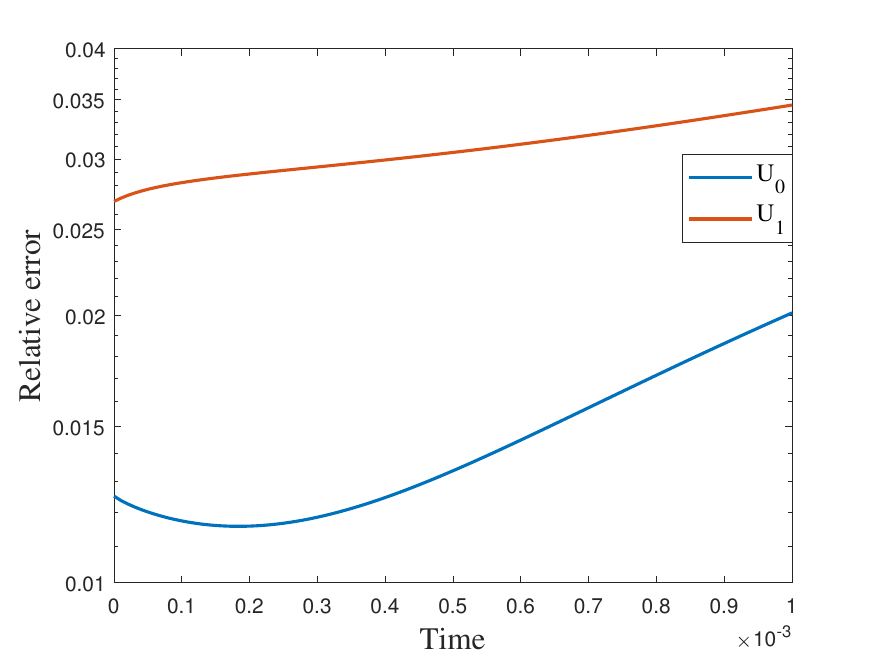}
        \end{minipage}
    }
    \\
    \subfloat
    {
        \begin{minipage}[t]{5cm}
            \centering
            \includegraphics[width=4.5cm,height=4cm]{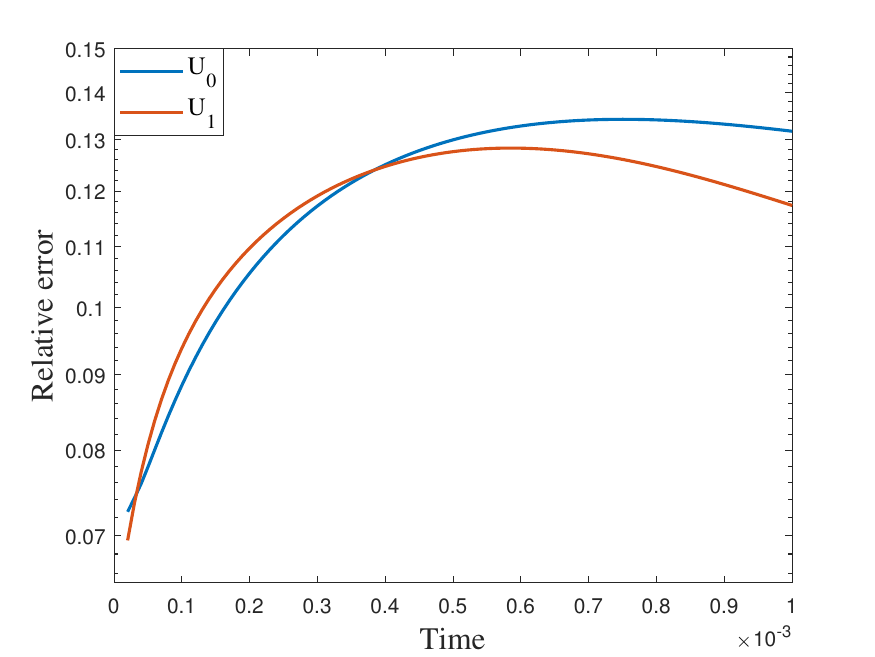}
        \end{minipage}
    }
    \subfloat
    {
        \begin{minipage}[t]{5cm}
            \centering
            \includegraphics[width=4.5cm,height=4cm]{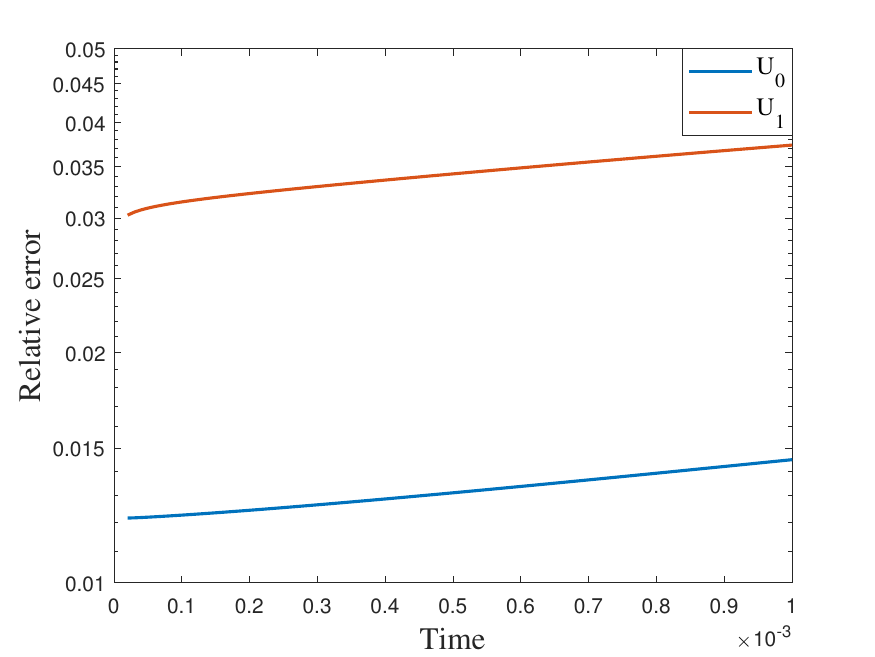}
        \end{minipage}
    }
    \subfloat
    {
        \begin{minipage}[t]{5cm}
            \centering
            \includegraphics[width=4.5cm,height=4cm]{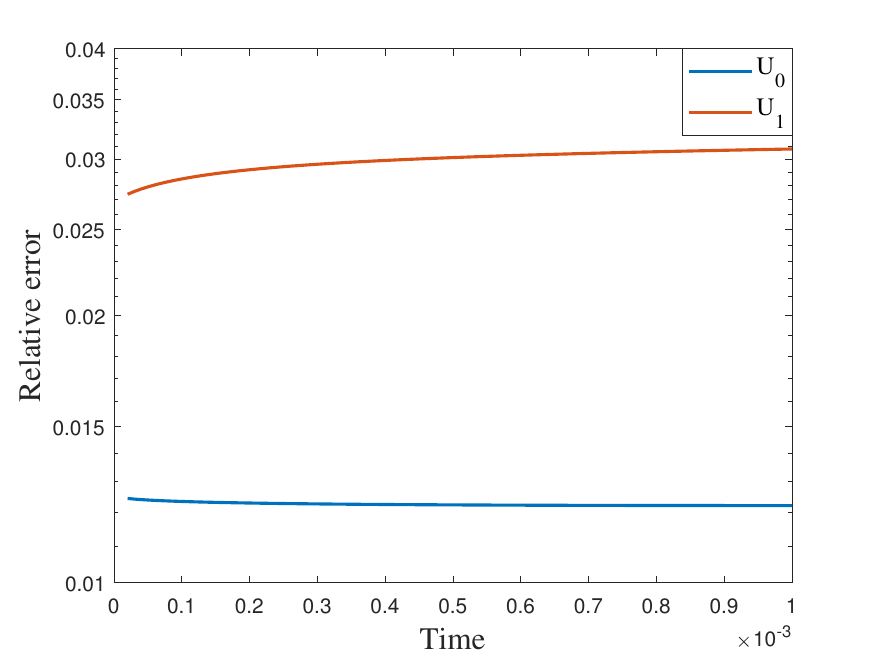}
        \end{minipage}
    }
    \caption{The first row: relative errors by explicit L1 scheme; The second row: relative errors by EI method. For each column, we take \(\alpha = 0.3, 0.6, 0.9\) form left to right.}
    \label{error_example3_k2_semi}
\end{figure}

\subsection{Rough or discontinuous media}
In our final example, we consider the medium in which the high-permeability region is more rough and discontinuous, as shown in Figure \ref{dis_media}. We set the initial condition \(u_{0}(x_{1},x_{2}) = -x_{1} (1-x_{1}) x_{2} (1-x_{2})\) and the source term \(f = 3u + t^{5}.\) The relative errors of our method are shown in Figure \ref{error_EI_discontinuous_kappa}. We would also illustrate the convergence rate of our method in this numerical experiment. Table \ref{table_error} displays the convergence rate for different fractional derivative. We note that to compute the convergence rate w.r.t. time, we keep the semidiscrete scheme the same (where we use the multicontinuum upscaled method in space) and take different time step size for the temporal discretization. We would evaluate the order of convergence by \(\log \bigl( e_{i}^{n}(\Delta t)/ e_{i}^{n}(\Delta t/2) \bigr), \ i=0,1\), where \(e_{i}^{n}(\Delta t)\) represents the relative error at \(t=T\) with respect to the time step \(\Delta t\).

\begin{figure}[H]
    \centering

    \includegraphics[width = 5cm]{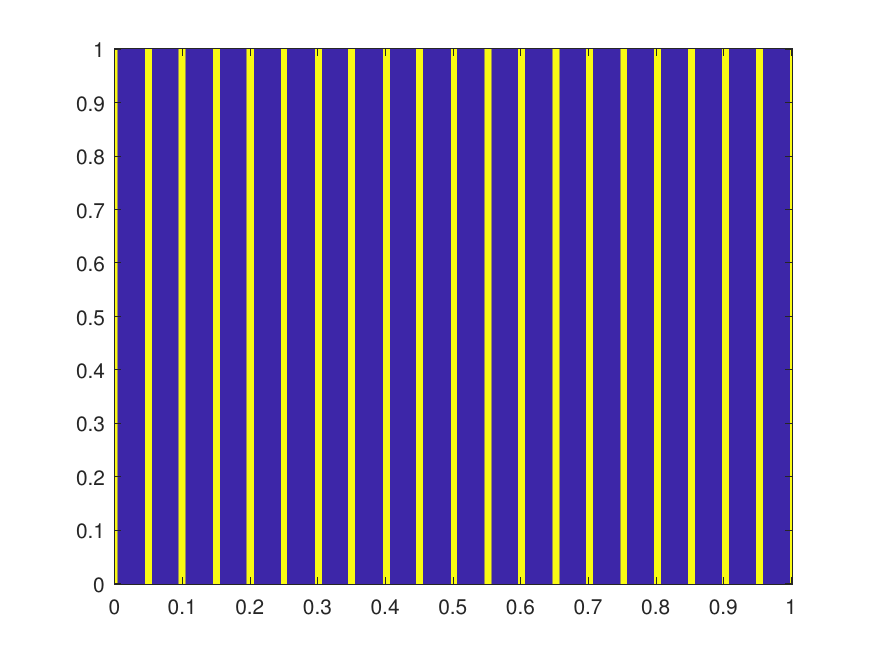}
    \caption{\(\kappa\) for a layered media.}
    \label{dis_media}
    
\end{figure}

\begin{figure}[H]
    \centering
    \subfloat
    {
        \begin{minipage}[t]{5cm}
            \centering
            \includegraphics[width=4.5cm,height=4cm]{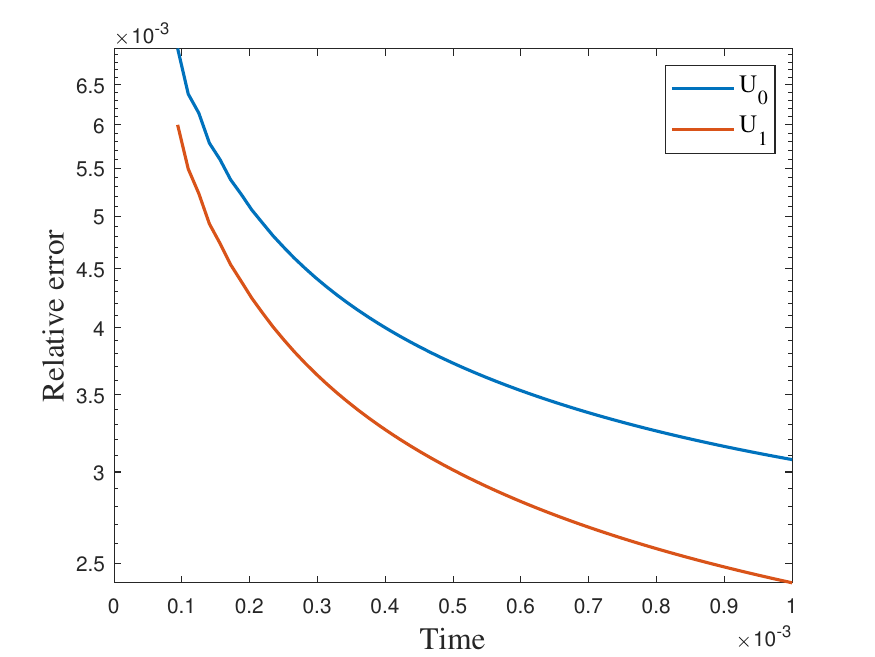}
        \end{minipage}
    }
    \subfloat
    {
        \begin{minipage}[t]{5cm}
            \centering
            \includegraphics[width=4.5cm,height=4cm]{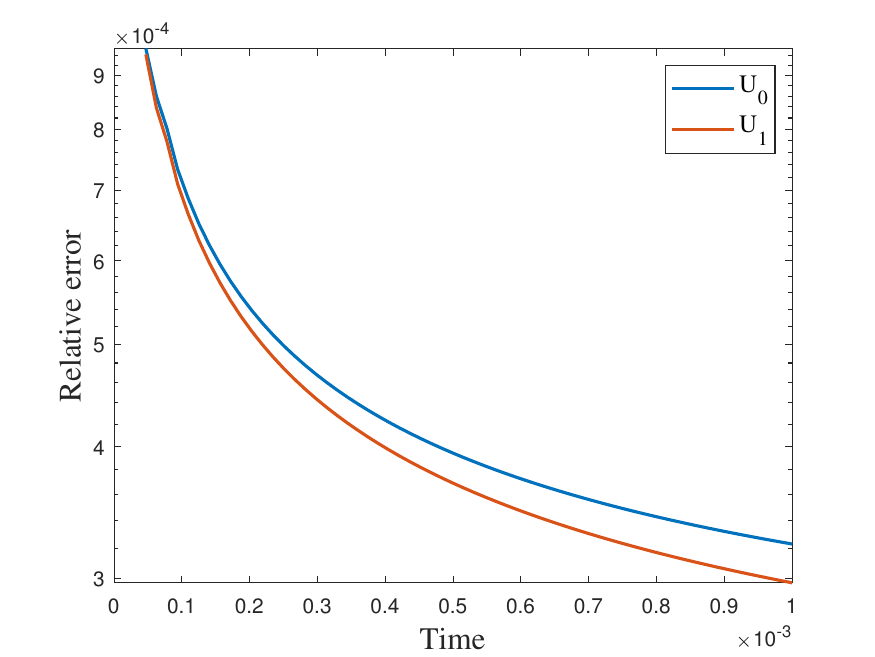}
        \end{minipage}
    }
    \subfloat
    {
        \begin{minipage}[t]{5cm}
            \centering
            \includegraphics[width=4.5cm,height=4cm]{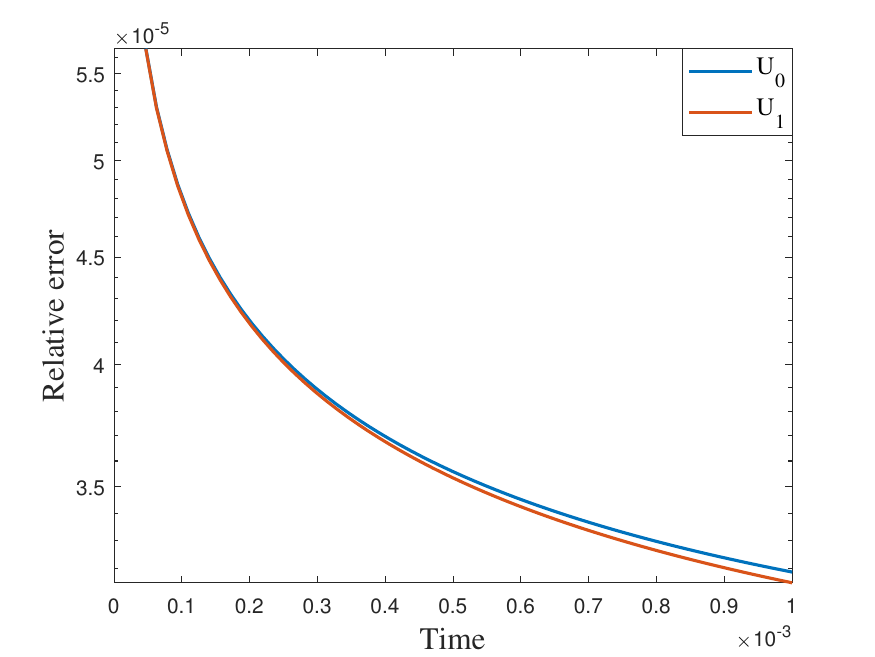}
        \end{minipage}
    }
    \caption{Relative errors by EI method. We take \(\alpha = 0.3, 0.5, 0.8\) form left to right.}
    \label{error_EI_discontinuous_kappa}
\end{figure}


\begin{table}[H]
    \belowrulesep=0pt
    \aboverulesep=0pt
    \centering
    \begin{tabular}{ccccccc}
        \toprule
        \(\alpha\)&0.3&0.4&0.5&0.6&0.7&0.8
        \\
        \midrule
        \(U_{0}\)&1.1301&1.4092&1.4404&1.6527&1.6706&1.7880\\
        \(U_{1}\)&1.2190&1.3430&1.4641&1.6314&1.6815&1.7879\\
        \bottomrule
    \end{tabular}
    \caption{Relative errors at \(t=T\).}
    \label{table_error}
\end{table}

\section{Conclusions}
In this paper, we present a time fractional multicontinuum upscaled model for solving the high-contrast fractional diffusion equation. The multicontinuum homogenization method is used to derive a coarse scale model, where the solutions are approximated through averages and accounts for gradient effects within each continuum, while incorporating microscale information in the upscaled quantities. We conduct a semi-discretized error analysis. We demonstrate that the nonlocal multicontinuum downscaling (NLMC) operator can be effectively approximated by the CEM-GMsFEM downscaling operator. Additionally, we utilize an exponential integrator approach for the time fractional derivative term and prove the convergence of the full discrete scheme. Numerical examples show that our method exhibits greater stability compared to the L1 scheme, while maintaining a similar convergence order. We explore various high-contrast heterogeneous media and source terms in our numerical experiments. It is noted that our proposed approach can be extended to scenarios involving two continua under appropriate assumptions. Future research would consist of addressing the singularity of the time fractional derivative and extended current approach to the high-order exponential integrator.

\section*{Acknowledgments}
Y. Wang's work is partially supported by the NSFC grant 12301559. W.T. Leung is partially supported by the Hong Kong RGC Early Career Scheme 21307223.

\bibliographystyle{unsrt}
\bibliography{nlmc_frac}

\begin{thebibliography}{10}

\bibitem{stynes2017error}
Martin Stynes, Eugene O'Riordan, and Jos{\'e}~Luis Gracia.
\newblock Error analysis of a finite difference method on graded meshes for a
  time-fractional diffusion equation.
\newblock {\em SIAM Journal on Numerical Analysis}, 55(2):1057--1079, 2017.

\bibitem{liao2018sharp}
Hong-lin Liao, Dongfang Li, and Jiwei Zhang.
\newblock Sharp error estimate of the nonuniform l1 formula for linear
  reaction-subdiffusion equations.
\newblock {\em SIAM Journal on Numerical Analysis}, 56(2):1112--1133, 2018.

\bibitem{jiang2017fast}
Shidong Jiang, Jiwei Zhang, Qian Zhang, and Zhimin Zhang.
\newblock Fast evaluation of the caputo fractional derivative and its
  applications to fractional diffusion equations.
\newblock {\em Communications in Computational Physics}, 21(3):650--678, 2017.

\bibitem{gu2020parallel}
Xian-Ming Gu and Shu-Lin Wu.
\newblock A parallel-in-time iterative algorithm for volterra partial
  integro-differential problems with weakly singular kernel.
\newblock {\em Journal of Computational Physics}, 417:109576, 2020.

\bibitem{efendiev2009multiscale}
Yalchin Efendiev and Thomas~Y Hou.
\newblock {\em Multiscale finite element methods: theory and applications},
  volume~4.
\newblock Springer Science \& Business Media, 2009.

\bibitem{hou1997multiscale}
Thomas~Y Hou and Xiao-Hui Wu.
\newblock A multiscale finite element method for elliptic problems in composite
  materials and porous media.
\newblock {\em Journal of computational physics}, 134(1):169--189, 1997.

\bibitem{hou1999convergence}
Thomas Hou, Xiao-Hui Wu, and Zhiqiang Cai.
\newblock Convergence of a multiscale finite element method for elliptic
  problems with rapidly oscillating coefficients.
\newblock {\em Mathematics of computation}, 68(227):913--943, 1999.

\bibitem{chen2003mixed}
Zhiming Chen and Thomas Hou.
\newblock A mixed multiscale finite element method for elliptic problems with
  oscillating coefficients.
\newblock {\em Mathematics of Computation}, 72(242):541--576, 2003.

\bibitem{dixon1985order}
Jennifer Dixon.
\newblock On the order of the error in discretization methods for weakly
  singular second kind non-smooth solutions.
\newblock {\em BIT Numerical Mathematics}, 25:623--634, 1985.

\bibitem{lubich1982runge}
Ch~Lubich.
\newblock Runge-kutta theory for volterra integrodifferential equations.
\newblock {\em Numerische Mathematik}, 40:119--135, 1982.

\bibitem{bourgeat1984homogenized}
Alain Bourgeat.
\newblock Homogenized behavior of two-phase flows in naturally fractured
  reservoirs with uniform fractures distribution.
\newblock {\em Computer Methods in Applied Mechanics and Engineering},
  47(1-2):205--216, 1984.

\bibitem{durlofsky1991numerical}
Louis~J Durlofsky.
\newblock Numerical calculation of equivalent grid block permeability tensors
  for heterogeneous porous media.
\newblock {\em Water resources research}, 27(5):699--708, 1991.

\bibitem{chen2003coupled}
Yuguang Chen, Louis~J Durlofsky, M~Gerritsen, and Xian-Huan Wen.
\newblock A coupled local--global upscaling approach for simulating flow in
  highly heterogeneous formations.
\newblock {\em Advances in water resources}, 26(10):1041--1060, 2003.

\bibitem{wu2002analysis}
Xiao-Hui Wu, Yalchin Efendiev, and Thomas~Y Hou.
\newblock Analysis of upscaling absolute permeability.
\newblock {\em Discrete and Continuous Dynamical Systems Series B},
  2(2):185--204, 2002.

\bibitem{efendiev2013generalized}
Yalchin Efendiev, Juan Galvis, and Thomas~Y Hou.
\newblock Generalized multiscale finite element methods (gmsfem).
\newblock {\em Journal of computational physics}, 251:116--135, 2013.

\bibitem{efendiev2011multiscale}
Yalchin Efendiev, Juan Galvis, and Xiao-Hui Wu.
\newblock Multiscale finite element methods for high-contrast problems using
  local spectral basis functions.
\newblock {\em Journal of Computational Physics}, 230(4):937--955, 2011.

\bibitem{chung2015residual}
Eric~T Chung, Yalchin Efendiev, and Wing~Tat Leung.
\newblock Residual-driven online generalized multiscale finite element methods.
\newblock {\em Journal of Computational Physics}, 302:176--190, 2015.

\bibitem{maalqvist2014localization}
Axel M{\aa}lqvist and Daniel Peterseim.
\newblock Localization of elliptic multiscale problems.
\newblock {\em Mathematics of Computation}, 83(290):2583--2603, 2014.

\bibitem{chung2018constraint}
Eric~T Chung, Yalchin Efendiev, and Wing~Tat Leung.
\newblock Constraint energy minimizing generalized multiscale finite element
  method.
\newblock {\em Computer Methods in Applied Mechanics and Engineering},
  339:298--319, 2018.

\bibitem{chung2018non}
Eric~T Chung, Yalchin Efendiev, Wing~Tat Leung, Maria Vasilyeva, and Yating
  Wang.
\newblock Non-local multi-continua upscaling for flows in heterogeneous
  fractured media.
\newblock {\em Journal of Computational Physics}, 372:22--34, 2018.

\bibitem{efendiev2023multicontinuum}
Yalchin Efendiev and Wing~Tat Leung.
\newblock Multicontinuum homogenization and its relation to nonlocal
  multicontinuum theories.
\newblock {\em Journal of Computational Physics}, 474:111761, 2023.

\bibitem{leung2024some}
Wing~Tat Leung.
\newblock Some convergence analysis for multicontinuum homogenization.
\newblock {\em arXiv preprint arXiv:2401.12799}, 2024.

\bibitem{contreras2023exponential}
Luis~F Contreras, David Pardo, Eduardo Abreu, Judit Mu{\~n}oz-Matute, Ciro
  D{\'\i}az, and Juan Galvis.
\newblock An exponential integration generalized multiscale finite element
  method for parabolic problems.
\newblock {\em Journal of Computational Physics}, 479:112014, 2023.

\bibitem{jin2019numerical}
Bangti Jin, Raytcho Lazarov, and Zhi Zhou.
\newblock Numerical methods for time-fractional evolution equations with
  nonsmooth data: a concise overview.
\newblock {\em Computer Methods in Applied Mechanics and Engineering},
  346:332--358, 2019.

\bibitem{li2024wavelet}
Guanglian Li.
\newblock Wavelet-based edge multiscale parareal algorithm for subdiffusion
  equations with heterogeneous coefficients in a large time domain.
\newblock {\em Journal of Computational and Applied Mathematics}, 440:115608,
  2024.

\bibitem{li2010existence}
Xianjuan Li and Chuanju Xu.
\newblock Existence and uniqueness of the weak solution of the space-time
  fractional diffusion equation and a spectral method approximation.
\newblock {\em Communications in Computational Physics}, 8(5):1016, 2010.

\bibitem{kemppainen2011existence}
Jukka Kemppainen.
\newblock Existence and uniqueness of the solution for a time-fractional
  diffusion equation with robin boundary condition.
\newblock In {\em Abstract and Applied Analysis}, volume 2011, page 321903.
  Wiley Online Library, 2011.

\bibitem{lin2007finite}
Yumin Lin and Chuanju Xu.
\newblock Finite difference/spectral approximations for the time-fractional
  diffusion equation.
\newblock {\em Journal of computational physics}, 225(2):1533--1552, 2007.

\bibitem{li2019constraint}
Mengnan Li, Eric Chung, and Lijian Jiang.
\newblock A constraint energy minimizing generalized multiscale finite element
  method for parabolic equations.
\newblock {\em Multiscale Modeling \& Simulation}, 17(3):996--1018, 2019.

\bibitem{podlubny1998fractional}
Igor Podlubny.
\newblock {\em Fractional differential equations: an introduction to fractional
  derivatives, fractional differential equations, to methods of their solution
  and some of their applications}.
\newblock elsevier, 1998.

\bibitem{yang2024accurate}
Zhengya Yang, Xuejuan Chen, Yanping Chen, and Jing Wang.
\newblock Accurate numerical simulations for fractional diffusion equations
  using spectral deferred correction methods.
\newblock {\em Computers \& Mathematics with Applications}, 153:123--129, 2024.

\bibitem{cox2002exponential}
Steven~M Cox and Paul~C Matthews.
\newblock Exponential time differencing for stiff systems.
\newblock {\em Journal of Computational Physics}, 176(2):430--455, 2002.

\bibitem{garrappa2011accurate}
Roberto Garrappa and Marina Popolizio.
\newblock On accurate product integration rules for linear fractional
  differential equations.
\newblock {\em Journal of Computational and Applied Mathematics},
  235(5):1085--1097, 2011.

\bibitem{garrappa2013exponential}
Roberto Garrappa.
\newblock Exponential integrators for time--fractional partial differential
  equations.
\newblock {\em The European Physical Journal Special Topics},
  222(8):1915--1927, 2013.

\bibitem{hosseini2017solution}
S~Mohammed Hosseini and Zohreh Asgari.
\newblock Solution of stochastic nonlinear time fractional pdes using
  polynomial chaos expansion combined with an exponential integrator.
\newblock {\em Computers \& Mathematics with Applications}, 73(6):997--1007,
  2017.

\bibitem{li2023exponential}
Buyang Li, Yanping Lin, Shu Ma, and Qiqi Rao.
\newblock An exponential spectral method using vp means for semilinear
  subdiffusion equations with rough data.
\newblock {\em SIAM Journal on Numerical Analysis}, 61(5):2305--2326, 2023.

\bibitem{li2021high}
Buyang Li and Shu Ma.
\newblock A high-order exponential integrator for nonlinear parabolic equations
  with nonsmooth initial data.
\newblock {\em Journal of Scientific Computing}, 87(1):23, 2021.

\bibitem{jin2018numerical}
Bangti Jin, Buyang Li, and Zhi Zhou.
\newblock Numerical analysis of nonlinear subdiffusion equations.
\newblock {\em SIAM Journal on Numerical Analysis}, 56(1):1--23, 2018.

\end{thebibliography}

\end{document}